\documentclass[11pt]{article}
\usepackage{amsmath,amsthm,amsfonts,amssymb,epsfig,enumerate}
\usepackage{flafter}
\usepackage{a4wide}
\usepackage{algorithm}
\usepackage{algorithmic}
\usepackage{authblk}

\newtheorem{theorem}{Theorem}[section]

 \newtheorem{lemma}{Lemma}[section]
 \newtheorem{remark}{Remark}
 \newtheorem{assumption}{Assumption}

\theoremstyle{remark}
\numberwithin{equation}{section}

\newcommand{\norm}[1]{\left\Vert#1\right\Vert}

\usepackage{graphicx}
\usepackage{float}
\usepackage{epstopdf}
\usepackage{multirow}
\usepackage{color}
\usepackage{array}
\usepackage{booktabs}
\usepackage{multirow}
\usepackage[export]{adjustbox}
\usepackage[colorlinks,linktocpage,linkcolor=blue]{hyperref}

\begin{document}
\title{
 Solving the inverse potential problem in the parabolic equation by the deep neural networks method
}

\author[1,2]{Mengmeng Zhang\thanks{mmzhang@hebut.edu.cn}}
\author[3]{Zhidong Zhang\thanks{zhangzhidong@mail.sysu.edu.cn}}
\affil[1]{\small{School of Science, Hebei University of Technology, Tianjin 300401,  China\\}}
\affil[2]{\small{Nanjing Center for Applied Mathematics\\
Nanjing, 211135, China\\}}
\affil[3]{\small{School of Mathematics (Zhuhai), Sun Yat-sen University, Zhuhai 519082, Guangdong, China\\}}

\maketitle


\begin{abstract}
In this work, we consider an inverse potential problem in the parabolic equation, where the unknown potential is a space-dependent function and the used measurement is the final time data. The unknown potential in this inverse problem is parameterized by deep neural networks (DNNs) for the reconstruction scheme. First, the uniqueness of the inverse problem is proved under some regularities assumption on the input sources. Then we propose a new loss function with regularization terms depending on the derivatives of the residuals for partial differential equations (PDEs) and the measurements. These extra terms effectively induce higher regularity in solutions so that the ill-posedness of the inverse problem can be handled. Moreover, we establish the corresponding generalization error estimates rigorously. Our proofs exploit the conditional stability of the classical linear inverse source problems, and the mollification on the noisy measurement data which is set to reduce the perturbation errors. Finally, the numerical algorithm and some numerical results are provided.
\end{abstract}

{\bf AMS subject classifications:} 34K28, 35R30, 65N15, 62M45.\\

{\bf Keywords:}\ inverse potential problem, deep neural networks, uniqueness, generalization error estimates, numerical reconstruction.

\section{Introduction.}\label{Sec_intro}
\subsection{Mathematical model.}

\par The following parabolic system is considered in this work:
\begin{equation}\label{eq1}
\begin{cases}
\begin{aligned}
(\partial_t -\Delta +q(x))u&=F(x,t), &&(x,t)\in \Omega_T,\\
u(x,t)&=b(x,t), &&(x,t)\in\partial\Omega_T,\\
u(x,0)&=u_0(x), &&x\in\Omega.
\end{aligned}
\end{cases}
\end{equation}
Here we write $\Omega_T=\Omega\times(0,T]$ and $\partial \Omega_T=\partial\Omega\times(0,T]$ for short, and $\Omega\subset \mathbb{R}^d$ is an open bounded domain in $\mathbb R^d$ with sufficiently smooth boundary. $F(x,t),\ u_0(x),\ b(x,t)$ are the source term, initial status, boundary condition respectively, causing the heat propagation in the medium. The potential function $q(x) \in L^{\infty}(\Omega)$, called the heat radiative coefficient of the material, is a crucial parameter for characterizing the heat conduction process. It describes the ability of the medium to propagate heat from internal sources or sinks. For known $(F(x,t),u_0(x),b(x,t), q(x))$ with suitable regularities, the forward problem (\ref{eq1}) is well-posed in appropriate function space \cite{Evans}. In this work, we consider the inverse problem of recovering the unknown $q(x)$, where the used measurement is the final time data
\begin{eqnarray}\label{eq2}
    u(x,T):=\varphi(x), \quad x\in\Omega.
\end{eqnarray}
%


In practical applications of inverse problems, the contamination on inverse problems is unavoidable. So we will be given the noisy data $\varphi^\delta$ instead of the exact data $\varphi(x)$ in (\ref{eq2}), which satisfies
\begin{eqnarray}\label{eq3}
\|\varphi^\delta-\varphi\|_{L^\infty(\Omega)}\leq \delta.
\end{eqnarray}
To handle the effect caused by the perturbations, people need to develop effective methods to improve the accuracy and robustness in applications.
In this study, we choose the deep neural networks (DNNs) to solve the inverse problem \eqref{eq1}-\eqref{eq3}. Comparing to traditional methods for solving inverse potential problem, this approach demonstrates the superiority in high-dimensional space and has the advantage of breaking the curse of dimensionality.

There are rare works on studying the inverse potential problem for parabolic equations using deep neural networks, especially the rigorous analysis of its convergence estimate. In this work, the authors will consider the solution of the inverse potential problem \eqref{eq1}-\eqref{eq3}  parameterized by DNNs for the reconstruction scheme. We propose a new loss function with regularization terms depending on the derivatives of the residuals for PDEs and measurements. The mollification method has been employed to improve the regularity of the noisy data. Also, the generalization error estimates are rigorously derived from the conditional stability of the linear inverse source problem and the mollification error estimate on noisy data.

\subsection{Literature.}

The reconstructions of $q(x)$ in \eqref{eq1} from some inversion input data have been studied extensively. For zero initial status, the uniqueness for $q(x)$ by
\eqref{eq1}-\eqref{eq2} is established in \cite{Prilepko}, while the unique reconstruction using final measurement data is studied in \cite{Prilepko1}.
In the case of non-zero initial status, the existence and uniqueness of the generalized solution $(u(x,t),q(x))\in W_p^{2,1}(\Omega_T) \times L^p(\Omega) $ with the time-average temperature measurement are given in \cite{Kamynin} for $(u_0,\varphi)$ with some regularities.
Choulli and Yamamoto \cite{Choulli} prove the generic well-posedness of the inverse problem in H\"{o}lder spaces by final measurement data, and then the conditional stability result in a Hilbert space setting for sufficiently small $T$ is studied in \cite{Choulli1}.
Chen et al \cite{ChenDH} consider the inverse potential problem from a partial measurements over $[T_0,T_1]\times \Omega$ with $[T_0,T_1]\subset [0,T]$, where the conditional stability estimates of the inverse problem in some Sobolev space and the reasonable convergence rates of the  Tikhonov regularization are derived.
Recently, Jin et al \cite{Jin1} uses the same observational data and shows a weighted $L^2$ stability in the standard $L^2$ norm under a positivity condition. They provide an error analysis of reconstruction scheme based on the standard output least-squares formulation with Tikhonov regularization (by an $H^1$-seminorm penalty).
Zhang et al \cite{Zhang-Zhou} prove the uniqueness of the identification from final time data for (sub)diffusion equation and show the conditional stability in Hilbert spaces under some suitable conditions on the problem data. The convergence and error analysis of the reconstruction discrete scheme are rigorously analyzed. The investigations in the inverse non-smooth potential problem are given in \cite{Zhangliu}, where the uniqueness for this nonlinear inverse problem is proved. Numerically, an iterative process called two-point gradient method is proposed by minimizing the  data-fit term and the penalty term alternatively, with a convergence analysis in terms of the tangential condition.
There also exists some works involving multiple coefficient identification. For example, Yamamoto and Zou \cite{zou1} investigate the simultaneous reconstruction of the initial temperature and heat radiative coefficient in a heat conductive system, with stability of the inverse problem and the reconstruction scheme. Kaltenbacher and Rundell \cite{Kaltenbacher2} consider the inverse problem of simultaneously recovering two unknowns, spatially dependent conductivity and the potential function from overposed data consisting of $u(x,T)$. The uniqueness result and the convergence of an iteration scheme are established. We also refer to
\cite{Lesnic,Chen,Deng1,Isakov2,Rundell,Lesnic1} and the references therein for the inverse potential problems in diffusion models from different types of observational data.

Recently, deep learning methods for solving PDEs have been realized as an effective approach, especially in high dimensional PDEs. Such methods have the advantage of breaking the curse of dimensionality. The basic idea is to use neural networks (nonlinear functions) to approximate the unknown solutions of PDEs by learning the parameters. For the forward problems, there exists many numerical works with deep neural networks involving the depth Ritz method (DRM) \cite{EWN-DRM}, the depth Galerkin method (DGM) \cite{Sirignano-DGM}, the DeepXDE method \cite{Lu-DeepXDE}, depth operator network method (DeepONet) \cite{Lu-DeepONet-2}, physical information neural networks (PINNs) \cite{Raissi}, the weak adversary neural network (WAN) \cite{Bao2, Bao1} and so on.
Theoretically, there are some rigorous analysis works investigating the convergence and error estimates for the solution of PDEs via neural networks, but the result are still far from complete. For example, the convergence rate of DRM with two layer networks and deep networks are  studied in \cite{Duan_2022,Hong_2021, Lu_2021,Luo_2020}; the convergence of PINNs is given in \cite{Jiao,Mishra2,de2023error,Shin1,Shin2}. For the inverse problems, the PINNs frameworks can be employed to solve the so-called data assimilation or unique continuation problems, and rigorous estimates on the generalization error of PINNs are established in \cite{Mishra2}.
Bao et al \cite{Bao1} develop the WAN to solve electrical impedance tomography (EIT) problem. In \cite{Zhang-Li-Liu}, the authors study a classical linear inverse source problem using the final time data under the frameworks of neural networks, where a rigorous generalization error estimate is proposed with a novel loss function including the Sobolev norm of some residuals. For more specific inverse problems applied in engineering and science, we refer to \cite{Chen-Lu,Raissi-Yazdani,Shukla}.

\subsection{Outline.}

The rest of this article is organized as follows. In Section \ref{sec_pre} we introduce the knowledge of neural networks and the setting of mollification. In Section \ref{sec_uni}, we introduce a conditional stability of the linear inverse source problem first. Then the uniqueness theorem (Theorem \ref{thm2}) of this inverse potential problem can be proved followed from the conditional stability.
In Section \ref{sec_main}, a novel loss function with specific regularization terms is introduced. Then we prove the generalization error estimates of data-driven solution of inverse problems, which is stated in Theorem \ref{main}.
In Section \ref{sec_num}, we propose the reconstruction algorithm and provide several experiments to show the validity of the proposed algorithm.

\section{Preliminaries.}\label{sec_pre}
\subsection{Neural network architecture.}
First we introduce the basic knowledge of neural network briefly. Note that $u_\theta$ and $q_\eta$ are two separate networks with different variables $(x,t)$ and $x$.
Thus, we use $\xi$ to denote collectively the network parameters for a parametric function $s_\xi(z)$
such that a general scheme can be applied for either $u_{\theta}(x,t)$ (with $z=(x,t),\ \xi=\theta$)
or $q_{\eta}(x)$ (with $z=x,\ \xi=\eta$).
For a positive integer $K\in \mathbb{N}$, a $K$-layer feed-forward neural
network of $s_\xi(z)$ for $z\in\mathbb{R}^{d_0}$ is a function $s_\xi(z)$ defined by
\begin{eqnarray}\label{DeepNN}
s_\xi(z):=W_{K} l_{K-1} \circ \cdots \circ l_{1}(z)+b_{K},
\end{eqnarray}
where the $k$-th layer $l_{k}: \mathbb{R}^{d_{k-1}} \rightarrow \mathbb{R}^{d_{k}}$
is given by $l_{k}(z)=\sigma\left(W_{k} z+b_{k}\right)$ with weights
$W_{k} \in \mathbb{R}^{d_{k} \times d_{k-1}}$ and biases
$b_{k} \in \mathbb{R}^{d_{k}}$ for $k=2, \cdots, K$. The activation function $\sigma(\cdot)$ includes sigmoid, tanh, ReLU (Rectified Linear Unit), softmax and so on \cite{Goodfellow-et-al-2016}. These activation functions introduce non-linearities and enable the network to learn complex patterns and relationships in the data.
 The neural network (\ref{DeepNN}) consists of an input layer with argument $z$, where $d_{0}=d$ is the problem dimension (also known as the size of input layer), an output layer which has the weights $W_{K} \in \mathbb{R}^{d_{K}\times d_{K-1}}$ and biases $b_{K} \in \mathbb{R}^{d_{K}}$, and $K-1$ hidden layers for some $K\in \mathbb{N}$.  The network parameters of all layers are collectively denoted by
\begin{eqnarray*}
\xi:=\left(W_{K}, b_{K}, W_{K-1}, b_{K-1}, \cdots, W_{1}, b_{1}\right).
\end{eqnarray*}
In Figure \ref{DNN}, we give a simple architectures
 of fully connected neural networks, where $z=(x_1,x_2,\cdots, x_d)$ is d-dimensional input variables, and the neural networks function is given as $s_\xi(z)=y_{NN}$.
\begin{figure}[H]
 \centering
 \includegraphics[width=0.6\textwidth,height=0.2\textheight]{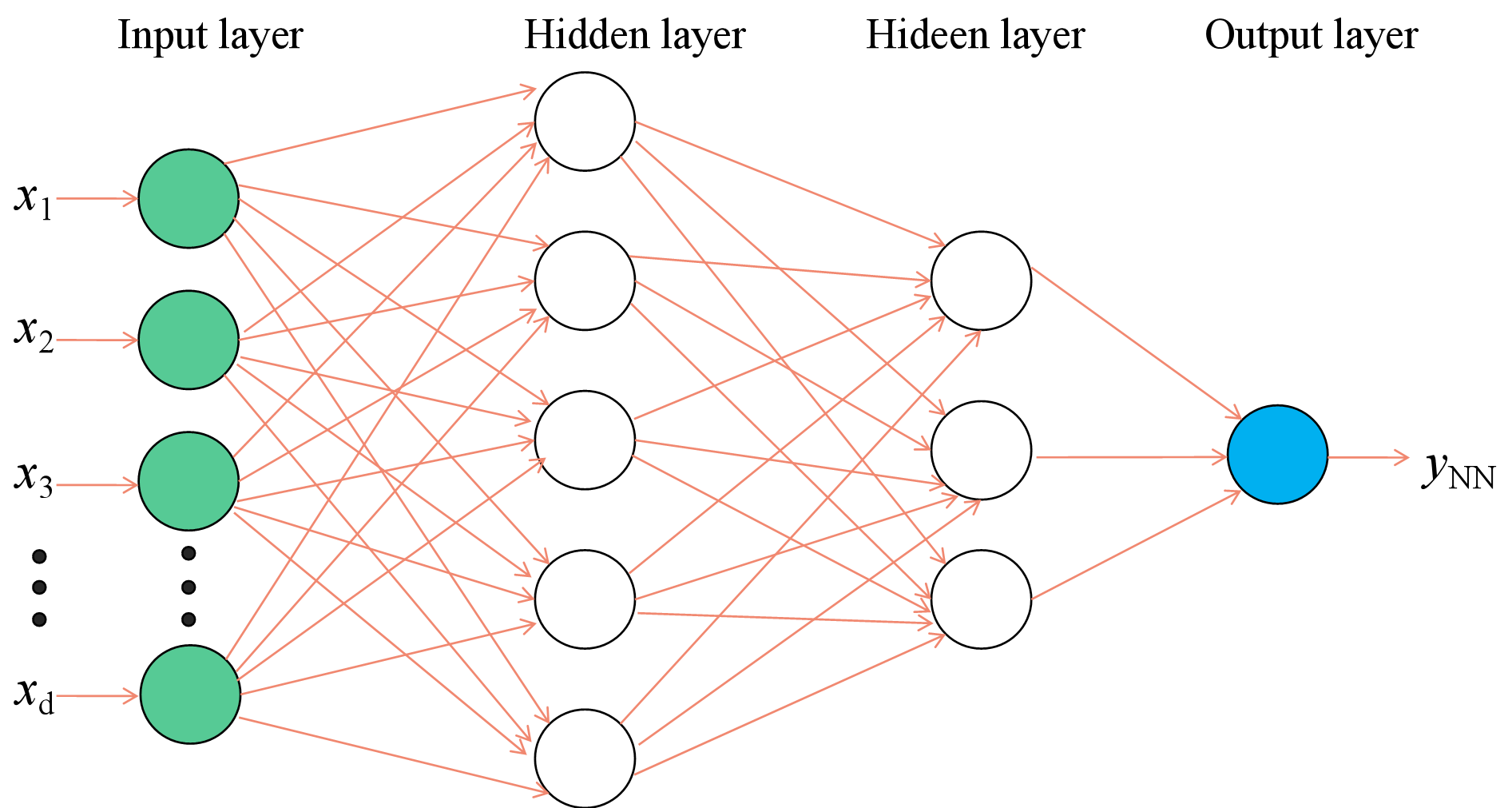}
 \caption
 {The fully connected neural networks.}\label{DNN}
 \end{figure}

\subsection{Mollification.}
In the practical applications of inverse problems, the noise of the measurements is unavoidable. The noisy data will make the residuals uncontrollable, which can be seen in the next section. Hence, we choose to mollify the measured data beforehand. The next is the introduction of mollification.

Fix one function $\rho\in C^2(\mathbb R)$ as
$$\text{supp}\rho=(0,1), \; \rho(0)=\rho(1)=\rho'(0)=\rho'(1)=0,$$
and
\begin{eqnarray*}
\int_0^{\infty} \rho(t) t^{d-1}\ dt=\frac{1}{\pi_d},
\end{eqnarray*}
with $\pi_d$ is the surface area of unit sphere $B(0,1)$ in $R^d$.
Set $\rho_\epsilon(x):=\epsilon^{-d}\rho(x/\epsilon)$, and define the mollifier as
\begin{equation}\label{mollifier}
 G_\epsilon \psi=\int_{|x-y|\le \epsilon} \rho_\epsilon(|x-y|)\psi(y)\ dy.
\end{equation}
Then we have
$$\int_{\mathbb R^d} \rho_\epsilon(|x-y|)\ dy=1. $$
In the next lemma, we concern with the estimate of $\Delta\varphi-\Delta G_\epsilon(\varphi^\delta)$.

\begin{lemma}\label{molli-lemma}
 Assume that the noisy data $\varphi^\delta\in L^\infty(\Omega)$ and the exact data $u(x,T):=\varphi(x)\in H^2(\Omega)$ satisfy
 \begin{eqnarray*}
\|\varphi-\varphi^\delta\|_{L^\infty(\Omega)}\le\delta.
\end{eqnarray*}
Also, the exact data imposes the high-order Lipschitz continuous condition. More precisely, we can find a positive constant $C_{\varphi}$ such that
\begin{align*}
 |\varphi(x)-\varphi(y)|&\leq C_{\varphi}|y-x|,\\
 |\Delta\varphi(x)-\Delta\varphi(y)|&\leq C_{\varphi}|y-x|,
\end{align*}
for $x,y\in \overline{\Omega}$ uniformly. For the mollification operator \eqref{mollifier}, if we pick $\epsilon=O(\delta^{1/3})$, then we can achieve the following optimal error bound
\begin{eqnarray*}
\|\Delta\varphi-\Delta G_\epsilon( \varphi^\delta)\|_{L^\infty(\Omega)}\le C\delta^{1/3}.
\end{eqnarray*}
\end{lemma}
\begin{proof}
 We split the subtraction $\Delta\varphi-\Delta G_\epsilon( \varphi^\delta)$ as following:
\begin{equation*}
 \Delta\varphi-\Delta G_\epsilon( \varphi^\delta)
 = (\Delta\varphi-G_\epsilon(\Delta \varphi))+(G_\epsilon(\Delta \varphi)-\Delta G_\epsilon( \varphi^\delta))=:I_1+I_2.
\end{equation*}
For $I_1$, we have that
\begin{equation*}
 |I_1|\le \int_{|x-y|\le \epsilon}\rho_\epsilon(|x-y|)~|\Delta\varphi(x)-\Delta\varphi(y)|\ dy
 \le C\epsilon.
\end{equation*}
For $I_2$, Green's identities and the properties of the kernel function $\rho$ give that
$\Delta G_\epsilon( \varphi)=G_\epsilon( \Delta\varphi)$. Hence,
\begin{equation*}
\begin{aligned}
 I_2&=\Delta\Big[ \int_{R^d}\rho_\epsilon(|x-y|)~(\varphi(y)-\varphi^\delta(y))  \ dy\Big]\\
 &=\int_{R^d}\Delta\rho_\epsilon(|x-y|)~(\varphi(y)-\varphi^\delta(y))\ dy.
 \end{aligned}
\end{equation*}
From the straightforward calculation, we can deduce that
\begin{equation*}
 \Delta\rho_\epsilon(|x-y|)=\epsilon^{-d-2} \rho''(|x-y|/\epsilon)+(d-1)\epsilon^{-d-1}|x-y|^{-1}\rho'(|x-y|/\epsilon),
\end{equation*}
which gives
\begin{equation*}
 |I_2|\le \delta \int_{|x-y|\le \epsilon}|\Delta\rho_\epsilon(|x-y|)|\ dy
 \le C \delta \epsilon^{-2}.
\end{equation*}
So we have
\begin{equation*}
 | \Delta\varphi-\Delta G_\epsilon(\varphi^\delta)|\le C\epsilon(1+\delta\epsilon^{-3}).
\end{equation*}
By picking $\epsilon=O(\delta^{1/3})$, we can achieve the desired estimate and complete the proof.
\end{proof}

\section{Uniqueness.}\label{sec_uni}
The uniqueness of this inverse potential problem is one of our main results. In this section, we will prove the uniqueness and the proof relies on the conditional stability of the inverse source problem of equation \eqref{eq1}. The conditional stability will be stated in the next subsection.

\subsection{Conditional stability of the inverse source problem.}
Under the framework of DNNs, the total error of the reconstructed solution depends on the training error and the measurement error. This connection relies on the conditional stability of linear inverse source problem, i.e., the quantitative dependence of the unknown source on the measurement data. Sequentially, here we will introduce some known results for the linear inverse source problem.

The mathematical statement of inverse source problem in parabolic equations with final time data is given below. For the parabolic equation
\begin{equation}\label{eqzz1}
 \begin{cases}
 \begin{aligned}
(\partial_t-\Delta +\overline{q}(x))v(x,t)&=p(x)h(x,t), &&(x,t)\in \Omega_T,\\
v(x,t)&=0, &&(x,t)\in \partial\Omega_T,\\
v(x,0)&=0, &&x\in\Omega,
\end{aligned}
\end{cases}
\end{equation}
we set $\overline q\ge 0$ and $\overline q\in L^\infty (\Omega)$, and $h(x,t)$ is given.
Then the inverse source problem is to use the measurement
\begin{eqnarray}\label{eqzz2}
\overline{\varphi}(x):=v[p](x,T)
\end{eqnarray}
to recover the unknown $p(x)$ in the source term.

Recalling the norm of the classical Sobolev space $W^{2,1}_2(\Omega_T)$ as
\begin{eqnarray*}
\|u\|_{W^{2,1}_2(\Omega_T)}=\sqrt{\sum\limits_{|\alpha|\leq 2}\|D^{\alpha}u\|^2_{L^2(\Omega_T)}+\|u_t\|^2_{L^2(\Omega_T)}},
\end{eqnarray*}
the following classical result on the inverse source problem \eqref{eqzz1}-\eqref{eqzz2} can be found in \cite{Prilepko}.
\begin{lemma}\label{lem3-0}
For equation \eqref{eqzz1}, we assume that
$$h\in L^\infty(\Omega_T),\ h_t\in L^\infty(\Omega_T),\ p(x)\in L^2(\Omega),\ ph\in L^2(\Omega_T),\ ph_t \in L^2(\Omega_T),$$
and
$$h(x,t)\ge 0,\ h_t(x,t)\geq 0\ \text{on}\ \Omega_T,\ |h(x,T)|\ge \nu >0\ \text{on}\ \Omega.$$
Here $\nu$ is a fixed positive number. Then, for known $\overline{q}\in L^\infty(\Omega)$ and input data $\overline{\varphi}\in H^2(\Omega)$, there exists a unique solution $(v(x,t), p(x)) \in W^{2,1}_2(\Omega_T)\times L^2(\Omega)$ to \eqref{eqzz1}-\eqref{eqzz2},
following the estimate
\begin{eqnarray*}\label{eqzz3}
\|p\|_{L^2(\Omega)}+\|v\|_{W^{2,1}_2(\Omega_T)}\leq C \|(-\Delta+\overline{q})\overline{\varphi}\|_{L^2(\Omega)}.
\end{eqnarray*}
The constant $C$ depends on $\norm{\overline q}_{L^\infty(\Omega)}$, $\nu,$ $\Omega$ and $T$.
\end{lemma}

\subsection{Uniqueness theorem.}
Now it is time to show the uniqueness theorem. First we introduce the admissible set for the unknown potential $q(x)$ as
$$\mathcal{A}:=\{\psi\in L^\infty(\Omega): 0\le \psi(x)\le M \text{ a.e. on }\Omega\}\subset L^2(\Omega).$$
The constant $M$ is the given upper bound of the admissible set. Next, recalling equation \eqref{eq1}, we collect some restrictions on the controllable source $F(x,t)$, initial status $u_0(x)$ and boundary condition $b(x,t)$.

\begin{assumption}\label{assumption}
 The assumptions on $F(x,t)$, $u_0(x)$ and $b(x,t)$ are given as follows.
\begin{itemize}
 \item $u_0(x)\in H^2(\Omega)$, $u_0(x)=b(x,0)$ on $\partial\Omega$, $\exists \nu>0$ such that  $u_0(x)\geq \nu >0$ on $\Omega$;
 \item $b\in H^2(\partial\Omega)$, $b\ge \nu>0$ on $\partial\Omega$, $b_t \geq 0$ on $\partial\Omega$;
 \item $F\in L^2(\Omega_T)$, $F_t\in L^2(\Omega_T)$, $F\ge 0$ on $\Omega_T$, $F_t\ge 0$ on $\Omega_T$;
 \item $\Delta u_0(x)-Mu_0(x)+F(x,0)\ge 0$ on $\Omega$.
\end{itemize}
\end{assumption}

\begin{theorem}\label{thm2}
Under Assumption \ref{assumption}, the inverse problem \eqref{eq1}-\eqref{eq3} has at most one solution in $W^{2,1}_2(\Omega_T)\times\mathcal A$.
\end{theorem}

\begin{proof}
Assume that there are two distinct pairs $(u[q_1], q_1)$ and $(u[q_2], q_2)$ satisfying \eqref{eq1}-\eqref{eq3} with same data
$$u[q_1](x,T)=u[q_2](x,T)=\varphi(x).$$
Setting $$w(x,t):=u[q_1](x,t)-u[q_2](x,t), \quad \overline q(x):=q_2(x)-q_1(x),$$
then $w(x,t)$ meets the system
\begin{equation}\label{dnn-uni1}
\begin{cases}
\begin{aligned}
(\partial_t-\Delta +q_1(x))w(x,t)&=\overline q(x)u[q_2](x,t), &&(x,t)\in \Omega_T,\\
w(x,t)&=0, &&(x,t)\in \partial\Omega_T,\\
w(x,0)&=0, &&x\in\Omega,
\end{aligned}
\end{cases}
\end{equation}
with
\begin{equation}\label{dnn-uni2}
w(x,T) = 0,  \quad  x\in\Omega.
\end{equation}
We need to prove that
$$(w(x,t), \overline q(x))=(0,0)\ \text{in}\ W^{2,1}_2(\Omega_T)\times L^\infty (\Omega).$$

Obviously $\overline q(x)\in L^2(\Omega)$. Also there holds $u[q_2]\in L^\infty(\Omega_T)$ and
$u_t[q_2]\in L^\infty(\Omega_T)$ by \cite[Lemma 2.1]{Zhangliu}. Then we have
$$\overline q u[q_2]\in L^2(\Omega_T),\quad \overline q u_t[q_2]\in L^2(\Omega_T).$$
Under Assumption \ref{assumption} and the maximum principle, we can see that $u[q_2]\geq 0$ on   $\Omega_T$. For $u_t[q_2]$, with Assumption \ref{assumption} and equation \eqref{eq1}, it satisfies
\begin{equation*}
\begin{cases}
\begin{aligned}
(\partial_t -\Delta +q_2(x)) (u_t[q_2])&=F_t(x,t)\ge0, &&(x,t)\in \Omega_T,\\
u_t[q_2](x,t)&=b_t(x,t)\ge 0, &&(x,t)\in\partial\Omega_T,\\
u_t[q_2](x,0)&=\Delta u_0(x)-q_2u_0(x)+F(x,0)\ge 0, &&x\in\Omega.
\end{aligned}
\end{cases}
\end{equation*}
Then the maximum principle leads to $u_t[q_2]\ge 0$ straightforwardly. With the positivity of $u_t[q_2]$, we derive that
$$u[q_2](x,t)=u_0(x)+\int_0^t \partial_s u[q_2](x,s) ds \geq u[q_2](x,0)\geq \nu >0, \ (x,t)\in \Omega_T,$$
which yields $u[q_2](x,T)\geq \nu>0$.
Now the conditions of Lemma \ref{lem3-0} are satisfied, and we conclude $(w(x,t),\overline  q(x))=(0,0)$ by applying Lemma \ref{lem3-0} on \eqref{dnn-uni1}-\eqref{dnn-uni2}. The proof is complete.
\end{proof}

\section{Generalization error estimates.}\label{sec_main}

In this section, we will discuss the error estimate of our approach for the inverse potential problem. Firstly, we introduce the corresponding residuals and define the loss function.

\subsection{Loss function and training errors.}\label{Sec2-1}
We propose a formulation of loss function for data-driven solutions of inverse problems, which can ensure the accuracy with the conditional stability of the given linear inverse source  problem. To achieve it, we define suitable residuals that measure the errors of the governed system and the input data.

Assume that the activation function is of $C^2$ regularity for the neural network $u_{\theta}$ defined by \eqref{DeepNN}, which leads to $u_{\theta} \in H^{2}(\overline\Omega\times [0, T])$. For the network parameters
$$\theta\in\Theta:=\{(W_{k}, b_{k})\}_{k=1}^{K} :W_{k} \in \mathbb{R}^{d_{k} \times d_{k-1}}, b_{k} \in \mathbb{R}^{d_{k}}\},$$
the set of all possible trainable parameters $u_{\theta}(x,t)$ up to its second order weak derivatives are bounded in $\overline\Omega\times [0,T]$ for any specific $\theta$. Similarly, noticing that $q_\eta(x)$ is the parametric neural network to approximate the potential function $q(x)$, we assume the activation function for the neural network $q_\eta(x)$ is of $L^\infty$ regularity such that $q_\eta(x)\in L^\infty(\Omega)$. We define
\begin{itemize}
\item Interior PDE residual
\begin{eqnarray*}\label{res1}
\mathcal{R}_{{int},\theta,\eta}(x, t):=\partial_{t} u_{\theta}(x, t)-\Delta u_{\theta}(x, t)+ q_\eta(x)u_{\theta}(x, t)-F(x,t), \quad (x,t) \in \Omega_T.
\end{eqnarray*}
\item Spatial boundary residual
\begin{eqnarray*}\label{res2}
\mathcal{R}_{sb,\theta}(x, t):=u_{\theta}(x, t)-b(x,t), \quad (x,t) \in \partial\Omega_T.
\end{eqnarray*}
\item Initial status residual
\begin{eqnarray*}\label{res3}
\mathcal{R}_{tb, \theta}(x):=u_{\theta}(x, 0)- u_0(x), \quad  x \in \Omega.
\end{eqnarray*}
\item Data residual
\begin{eqnarray}\label{res4}
\mathcal{R}_{d, \theta}(x):=u_{\theta}(x, T)- G_\epsilon\varphi^\delta(x), \quad  x \in \Omega.
\end{eqnarray}
\end{itemize}

Note that in the data residual \eqref{res4}, we use the mollified data  $G_\epsilon\varphi^\delta(x)$ instead of the noisy data $\varphi^\delta(x)$. A loss function minimization scheme for data-driven inverse problems seeks to minimize these residuals  comprehensively with some weights balancing different residuals. The loss function is defined as follows:
\begin{equation}\label{Loss1}
\begin{aligned}
J_\lambda(\theta,\eta)
=& \|q_\eta \mathcal{R}_{d,\theta}\|^2_{L^2(\Omega)} +\|\Delta\mathcal{R}_{d,\theta}\|^2_{L^2(\Omega)}
+ \lambda \|\mathcal{R}_{{int},\theta,\eta}\|^2_{H^1(0,T;L^2(\Omega))} \\
&
+\|\mathcal{R}_{{tb},\theta}\|^2_{L^{2}(\Omega)}
+\|q_\eta\mathcal{R}_{{tb},\theta}\|^2_{L^{2}(\Omega)}
+\|\Delta\mathcal{R}_{{tb},\theta}\|^2_{L^{2}(\Omega)}+\|\mathcal{R}_{{sb},\theta}\|^2_{H^2(0,T;L^2(\partial\Omega))},\\
\end{aligned}
\end{equation}
where $\lambda$ is a hyper-parameter to balance the residuals between the knowledge of PDE and the measurements. The proposed loss function (\ref{Loss1}) includes derivative penalties on the residuals. This is motivated by the conditional stability result for linear inverse source problem, which requires higher regularity on the measurement data $u(\cdot,T)$ (see Lemma \ref{lem3-0}). To improve the regularity of the noisy measurement data, we employ the mollification method by applying the mollification operator $G_\varepsilon$ on the noisy data $\varphi^\delta$. The design of the loss function for inverse problems distinguishes itself from that for forward problems such as physics-informed neural networks. The smoothness requirements not only ensure the existence of forward problem solutions, but also ensure the well-posedness of the inverse problem within the optimization framework.

\begin{remark}
The following standard loss function
\begin{eqnarray}\label{Loss-general}
 J^s(\theta,\eta)
=  \|\mathcal{R}_{d,\theta}\|_{L^2(\Omega)}^2
+\lambda \|\mathcal{R}_{{int},\theta,\eta}\|_{L^2(\Omega_T)}^2+
 \|\mathcal{R}_{{tb},\theta}\|_{L^{2}(\Omega)}^2
+\|\mathcal{R}_{{sb},\theta}\|_{L^2(\partial\Omega_T)}^2
\end{eqnarray}
has often been used in the literature.
For example, the DGM workflow adopts this form of loss function and minimizes it by least squares scheme \cite{Sirignano-DGM}.

\end{remark}

To determine $(\theta,\eta)$ from the discrete training set, accurate numerical evaluation of the integrals in (\ref{Loss1}) is essential. We introduce the following training sets that facilitate efficient computation of the integrals, leading to better performance:
\begin{align*}
\mathcal{S}_{d}&:=\left\{(x_n,T): x_n\in \Omega,\quad n=1,2,\cdots,N_{d} \right\},\nonumber \\
\mathcal{S}_{int}&:=\left\{(\widetilde{x}_{n},\widetilde{t}_{n}): (\widetilde{x}_{n},\widetilde{t}_{n})\in \Omega_T, \quad n=1,2,\cdots,N_{int}\right\}, \nonumber \\
\mathcal{S}_{tb}&:=\left\{(\overline{x}_{n},0): \overline{x}_{n}\in \Omega, \quad n=1,2,\cdots,N_{tb}\right\},\\
\mathcal{S}_{sb}&:=\left\{(\widehat{x}_{n},\widehat{t}_{n}): (\widehat{x}_{n},\widehat{t}_{n})\in \partial\Omega_T,\quad n=1,2,\cdots,N_{sb} \right\}.\nonumber
\end{align*}
Applying these sets and the numerical quadrature rules \cite{Mishra2}, we get the following empirical loss function
\begin{equation}\label{Loss1-1}
\begin{aligned}
 J_\lambda^N(\theta,\eta)
&=\sum_{n=1}^{N_{d}}\omega_n^{d,0}|q_\eta(x_n)\mathcal{R}_{d,\theta}(x_n)|^2+
\sum_{n=1}^{N_{d}}\omega_n^{d,1}|\Delta\mathcal{R}_{d,\theta}(x_n)|^2\\
&\quad+\lambda \sum_{n=1}^{N_{int}}\omega_n^{int,0}|\mathcal{R}_{{int},\theta,\eta}(\widetilde{x}_{n}, \widetilde{t}_{n})|^2
+\lambda \sum_{n=1}^{N_{int}}\omega_n^{int,1}|\partial_t\mathcal{R}_{{int},\theta,\eta}(\widetilde{x}_{n}, \widetilde{t}_{n})|^2 \\
&\quad+\sum_{n=1}^{N_{tb}}\omega_n^{tb,0}|\mathcal{R}_{tb,\theta}(\overline{x}_{n})|^2
+ \sum_{n=1}^{N_{tb}}\omega_n^{tb,1}|q_\eta(\overline{x}_n)\mathcal{R}_{tb,\theta}(\overline{x}_{n})|^2
+\sum_{n=1}^{N_{tb}}\omega_n^{tb,2}|\Delta\mathcal{R}_{tb,\theta}(\overline{x}_{n})|^2\\
&\quad+\sum_{n=1}^{N_{sb}}\omega_n^{sb,0}|\mathcal{R}_{sb,\theta}(\widehat{x}_{n}, \widehat{t}_{n})|^2
+\sum_{n=1}^{N_{sb}}\omega_n^{sb,1}|\partial_t\mathcal{R}_{sb,\theta}(\widehat{x}_{n}, \widehat{t}_{n})|^2 \sum_{n=1}^{N_{sb}}\omega_n^{sb,2}|\partial_t^2\mathcal{R}_{sb,\theta}(\widehat{x}_{n}, \widehat{t}_{n})|^2, \quad\quad\quad
\end{aligned}
\end{equation}
where the coefficients
$$\omega^{d,k}_n,\ \omega^{int,k}_n,\ \omega^{tb,j}_n,\ \omega^{sb,j}_n,\ k=0,1,\ j=0,1,2$$
are the quadrature weights. It is easy to see that the error for the loss function is
\begin{eqnarray}\label{Loss-error}
|J_\lambda(\theta,\eta)-J_\lambda^N(\theta,\eta)|
\le C\min\{ N_{d}^{-\alpha_{d,k}},N_{int}^{-\alpha_{int,k}},N_{tb}^{-\alpha_{tb,j}},
N_{sb}^{-\alpha_{sb,j}}:k=0,1,\ j=0,1,2\},\quad
\end{eqnarray}
where $C$ depends on the continuous norm $\|\cdot\|_{C(\Omega)}$ of the integrals, the rate $\alpha^{d,k}$, $ \alpha^{int,k}$, $\alpha^{tb,j}$, $\alpha^{sb,j}$ ($k=0,1,\ j=0,1,2$) are positive and depend on the regularity of the underlying integrand i.e, on the space $C(\Omega)$. Therefore, the underlying solutions and neural networks should be sufficiently regular such that the residuals can be approximated to a high accuracy by the quadrature rule.

Now, we define the generalization errors as
\begin{eqnarray}\label{err1}
\begin{cases}
\mathcal{E}_{G,q}:=\left\|q-q^*\right\|_{L^2(\Omega)},\\
\mathcal{E}_{G,u}:=\left\|u-u^{*}\right\|_{C\left([0, T] ;
L^{2}(\Omega)\right)},
\end{cases}
\end{eqnarray}
where $u^{*}:=u_{\theta^{*}},\ q^*:=q_{\eta^*}$ with $(\theta^*,\eta^*)$ is the minimizer of the functional (\ref{Loss1-1}). Also, we estimate generalization errors in terms of the following  training errors:
\begin{itemize}
\item The measurement data training errors:
$\mathcal{E}_{T,d}:=\mathcal{E}_{T,d,0}+\mathcal{E}_{T,d,1}$, where
\begin{eqnarray}\label{training1}
\begin{cases}
\mathcal{E}_{T,d,0}:=\left(\sum_{n=1}^{N_{d}} \omega_{j}^{d,0}\left|
q_\eta\left(x_{n}\right)\mathcal{R}_{d,\theta^{*}}\left(x_{n}\right)\right|^{2}\right)^{\frac{1}{2}},\\
\mathcal{E}_{T,d,1}:=\left(\sum_{n=1}^{N_{d}} \omega_{j}^{d,1}\left|\Delta \mathcal{R}_{d,\theta^{*}}\left(x_{n}\right)\right|^{2}\right)^{\frac{1}{2}}.
\end{cases}
\end{eqnarray}
\item The interior PDE training errors: $\mathcal{E}_{T,int}:=\mathcal{E}_{T,int,0}+\mathcal{E}_{T,int,1}$, where
\begin{eqnarray}\label{training2}
\begin{cases}
\mathcal{E}_{T,int,0}:=\left(\sum_{n=1}^{N_{int}}\omega_n^{int,0}|\mathcal{R}_{{int},\theta^{*},\eta^*}(\widetilde{x}_{n}, \widetilde{t}_{n})|^2\right)^{\frac{1}{2}},\\
\mathcal{E}_{T,int,1}:=\left(\sum_{n=1}^{N_{int}}\omega_n^{int,1}|\partial_t\mathcal{R}_{{int},\theta^{*},\eta^*}(\widetilde{x}_{n}, \widetilde{t}_{n})|^2\right)^{\frac{1}{2}}.
\end{cases}
\end{eqnarray}
\item The initial condition training errors: $\mathcal{E}_{T,tb}:=\mathcal{E}_{T,tb,0}+\mathcal{E}_{T,tb,1}+\mathcal{E}_{T,tb,2}$, where
\begin{eqnarray}\label{training4}
\begin{cases}
\mathcal{E}_{T,tb,0}
:=\left(\sum_{n=1}^{N_{tb}}\omega_n^{tb,0}|\mathcal{R}_{tb,\theta^{*}}(\overline{x}_{n})|^2\right)^{\frac{1}{2}},
\\
\mathcal{E}_{T,tb,1}
:=\left(\sum_{n=1}^{N_{tb}}\omega_n^{tb,1}|q_\eta(\overline{x}_{n})\mathcal{R}_{tb,\theta^{*}}(\overline{x}_{n})|^2\right)^{\frac{1}{2}},
\\
\mathcal{E}_{T,tb,2}
:=\left(\sum_{n=1}^{N_{tb}}\omega_n^{tb,2}|\Delta\mathcal{R}_{tb,\theta^{*}}(\overline{x}_{n})|^2\right)^{\frac{1}{2}}.
\end{cases}
\end{eqnarray}
\item The spatial boundary condition training errors: $\mathcal{E}_{T,sb}:=\mathcal{E}_{T,sb,0}+\mathcal{E}_{T,sb,1}+\mathcal{E}_{T,sb,2}$, where
\begin{eqnarray}\label{training3}
\begin{cases}
\mathcal{E}_{T,sb,0}
:=\left(\sum_{n=1}^{N_{sb}}\omega_n^{sb,0}|\mathcal{R}_{sb,\theta^{*}}(\widehat{x}_{n}, \widehat{t}_{n})|^2\right)^{\frac{1}{2}},\\
\mathcal{E}_{T,sb,1}
:=\left(\sum_{n=1}^{N_{sb}}\omega_n^{sb,1}|\partial_t\mathcal{R}_{sb,\theta^{*}}(\widehat{x}_{n}, \widehat{t}_{n})|^2\right)^{\frac{1}{2}},\\
\mathcal{E}_{T,sb,2}
:=\left(\sum_{n=1}^{N_{sb}}\omega_n^{sb,2}|\partial_t^2\mathcal{R}_{sb,\theta^{*}}(\widehat{x}_{n}, \widehat{t}_{n})|^2\right)^{\frac{1}{2}}.
\end{cases}
\end{eqnarray}
\end{itemize}

\subsection{Proofs of the estimates.}
Now we can state the theorem about the generalization error estimates.
\begin{theorem}\label{main}
Recall the errors defined in \eqref{err1}-\eqref{training3}. Under Assumption \ref{assumption}, there exists a unique solution to the inverse problem \eqref{eq1}-\eqref{eq3}. Moreover, for the approximate solution $(u^*,q^*)$ of the inverse problem with $(\theta^*,\eta^*)$ being a global minimizer of the loss function  $J_\lambda^N(\theta,\eta)$, we have the following generalization error estimates
\begin{equation}\label{IPf}
 \begin{aligned}
  \mathcal{E}_{G,q}&\leq C\Big( \mathcal{E}_{T,d}+ \mathcal{E}_{T,int}
+ \mathcal{E}_{T,sb,1}
+ \mathcal{E}_{T,sb,2}
+ \mathcal{E}_{T,tb,1}
+ \mathcal{E}_{T,tb,2}
+C_{q}^{\frac{1}{2}} N^{\frac{-\alpha}{2}}
+O(\delta^{1/3})\Big),\\
\mathcal{E}_{G,u}&\leq
C\Big( \mathcal{E}_{T,d}+ \mathcal{E}_{T,int}
+ \mathcal{E}_{T,sb}
+ \mathcal{E}_{T,tb}
+C_{q}^{\frac{1}{2}} N^{\frac{-\alpha}{2}}
+O(\delta^{1/3})\Big),
 \end{aligned}
\end{equation}
where
\begin{align*}
N&=\min \left\{N_{d}, N_{int},N_{sb},N_{tb}\right\},\\
\alpha&=\min \left\{\alpha_{int,0},\alpha_{int,1},\alpha_{sb,0},\alpha_{sb,1},\alpha_{sb,2},\alpha_{tb,0},\alpha_{tb,1},\alpha_{d}\right\},\\
\end{align*}
in \eqref{Loss-error}, and
$$C_{q}=\max\left\{C_{q,0},C_{q,1},
C_{qs,0},C_{qs,1},C_{qs,2},C_{qt,0},C_{qt,1}, C_{qd} \right\},$$
with
\begin{align*}
&C_{qd}=C_{qd}(\|\mathcal{L^*}\mathcal{R}_{d,\theta^*}\|_{C(\Omega)}),
&&C_{q,0}=C_{q,0}(\|\mathcal{R}_{int,\theta^*,\eta^*}\|_{C(\Omega_T)}),\\
&C_{q,1}=C_{q,1}(\|\partial_t\mathcal{R}_{int,\theta^*,\eta^*}\|_{C(\Omega_T)}),
&&C_{qs,0}=C_{qs,0}(\|\mathcal{R}_{sb,\theta^*}\|_{C(\partial\Omega_T)}),\\
&C_{qs,1}=C_{qs,1}(\|\partial_t\mathcal{R}_{sb,\theta^*}\|_{C(\partial\Omega_T)}), &&C_{qs,2}=C_{qs,2}(\|\partial_t^2\mathcal{R}_{sb,\theta^*}\|_{C(\partial\Omega_T)}),\\
&C_{qt,0}=C_{qt,0}(\|\mathcal{R}_{tb,\theta^*}\|_{C(\Omega)}), &&C_{qt,1}=C_{qt,1}(\|\mathcal{L^*}\mathcal{R}_{tb,\theta^*}\|_{C(\Omega)}).
\end{align*}
The constant $C$ depends on $\|q^*\|_{L^\infty(\Omega)},$ $\Omega$ and $T$.
\end{theorem}

\begin{proof}
First, we introduce $\hat u:=u^*-u$ and realize that
\begin{equation*}\label{sensitivitypro}
\begin{cases}
\begin{aligned}
(\partial_t -\Delta +q^*(x))\hat u(x,t) &=\mathcal{R}_{{int},\theta^*,\eta^*}(x, t)+(q-q^*)u[q](x,t), &&(x,t)\in \Omega_T,\\
\hat u(x,t)&=\mathcal{R}_{sb,\theta^*}(x,t), &&(x,t)\in\partial\Omega_T,\\
\hat u(x,0)&=\mathcal{R}_{tb,\theta^*}(x), && x\in\Omega,
\end{aligned}
\end{cases}
\end{equation*}
with the final condition
\begin{eqnarray*}\label{sensitivityObs}
\hat u(x,T)=u[q^*](x,T)-u[q](x,T)=\mathcal{R}_{d,\theta^*}(x)-(\varphi-G_\epsilon\varphi^\delta).
\end{eqnarray*}
We make the decomposition $\hat u:=\hat u_1+\hat u_2$, where $\hat u_1,\ \hat u_2$ satisfy
\begin{equation}\label{sensitivitypro1}
\begin{cases}
\begin{aligned}
(\partial_t -\Delta +q^*(x))\hat u_1(x,t)&=(q^*-q)(x)u[q](x,t), &&(x,t)\in \Omega_T,\\
\hat u_1(x,t)&=0, &&(x,t)\in\partial\Omega_T,\\
\hat u_1(x,0)&=0, &&x\in\Omega,
\end{aligned}
\end{cases}
\end{equation}
with
\begin{eqnarray}\label{sensitivityObs1}
\hat u_1(x,T)=\mathcal{R}_{d,\theta^*}(x)-(\varphi(x)-G_\epsilon\varphi^\delta(x))-\hat u_2(x,T),
\end{eqnarray}
and
\begin{equation*}\label{sensitivitypro2}
\begin{cases}
\begin{aligned}
(\partial_t -\Delta + q^*(x))\hat u_2(x,t)&=\mathcal{R}_{{int},\theta^{*},\eta^*}(x, t), &&(x,t)\in \Omega_T,\\
\hat u_2(x,t)&=\mathcal{R}_{sb,\theta^*}(x,t), &&(x,t)\in\partial\Omega_T,\\
\hat u_2(x,0)&=\mathcal{R}_{tb,\theta^*}(x), &&x\in\Omega,
\end{aligned}
\end{cases}
\end{equation*}
respectively.

Define the operator $\mathcal L^*$ as $\mathcal L^* \psi= (-\Delta+q^*)\psi.$
With Assumption \ref{assumption}, we can apply Lemma \ref{lem3-0} to (\ref{sensitivitypro1})-(\ref{sensitivityObs1}) and deduce that
\begin{equation}\label{ff}
\begin{aligned}
&\|q^{*}-q\|_{L^2(\Omega)}\\
&\leq C\|\mathcal{L}^*\hat u_1(\cdot,T)\|_{L^2(\Omega)}\\
&= C\|\mathcal{L}^* \mathcal{R}_{d,\theta^*}-\mathcal{L}^*(\varphi-G_\epsilon\varphi^\delta)-\mathcal{L}^*\hat u_2(\cdot,T)\|_{L^2(\Omega)}\\
&= C\|\mathcal{L}^* \mathcal{R}_{d,\theta^*}+\Delta \varphi-\Delta G_\epsilon\varphi^\delta-q^*(\varphi-G_\epsilon\varphi^\delta)-\mathcal{L}^*\hat u_2(\cdot,T)\|_{L^2(\Omega)}\\
&\leq C \left(\|\mathcal{L}^*\mathcal{R}_{d,\theta^*}\|_{L^2(\Omega)}+\|\Delta \varphi-\Delta G_\epsilon\varphi^\delta\|_{L^2(\Omega)}+\|q^*(\varphi-G_\epsilon\varphi^\delta)\|_{L^2(\Omega)}+ \|\mathcal{L}^*\hat u_2(\cdot,T)\|_{L^2(\Omega)}\right),
\end{aligned}
\end{equation}
with $C=C(\|q^*\|_{L^\infty(\Omega)},\Omega,T)$. Using Lemma \ref{molli-lemma}, we get
\begin{equation}\label{inequ-1}
\|\Delta \varphi-\Delta G_\epsilon\varphi^\delta\|_{L^2(\Omega)}
\leq  C\epsilon(1+\delta\epsilon^{-3}).
\end{equation}
Also, we have
\begin{align*}
|(\varphi-G_\epsilon\varphi^\delta)|
&\le
\int_{|x-y|\le \epsilon}\rho_\epsilon(|x-y|)~|\varphi(x)-\varphi^\delta(y)|\;dy  \\
&\le\int_{|x-y|\le \epsilon} \rho_\epsilon(|x-y|)~|\varphi(x)-\varphi(y)|dy + \int_{|x-y|\le \epsilon} \rho_\epsilon(|x-y|)~|\varphi(y)-\varphi^\delta(y)|\;dy \\
& \le C\epsilon + \delta .
\end{align*}
Thus, there holds
\begin{equation}\label{inequ-2}
\|q^*(\varphi-G_\epsilon\varphi^\delta)\|_{L^2(\Omega)}
\leq C\epsilon + \delta.
\end{equation}
By straightforward computations, we have
\begin{align}
 \|\mathcal{L}^*\hat u_2(\cdot,T)\|_{L^2(\Omega)}
 &\leq \|\partial_t\hat u_2(\cdot,T)\|_{L^2(\Omega)}+
 \|\mathcal{R}_{{int},\theta^{*},\eta^*}(\cdot, T)\|_{L^2(\Omega)}\nonumber \\
 &\leq \|\partial_t\hat u_2\|_{L^\infty(0,T;L^2(\Omega))}+
 \|\mathcal{R}_{{int},\theta^{*},\eta^*}(\cdot, T)\|_{L^2(\Omega)}.
\end{align}
Setting $w(x,t):=\partial_t \hat u_2(x,t)$, it satisfies
\begin{equation}\label{sensitivitypro3}
\begin{cases}
\begin{aligned}
(\partial_t -\Delta +q^*)w(x,t)&=\partial_t\mathcal{R}_{{int},\theta^*,\eta^*}(x, t), &&(x,t)\in \Omega_T,\\
w(x,t)&=\partial_t \mathcal{R}_{sb,\theta^*}(x,t), &&(x,t)\in\partial\Omega_T,\\
w(x,0)&=\mathcal{R}_{{int},\theta^*,\eta^*}(x, 0)-\mathcal{L}^*\mathcal{R}_{tb,\theta^*}(x), &&x\in\Omega.
\end{aligned}
\end{cases}
\end{equation}
Using the regularity theory for the direct problem (\ref{sensitivitypro3}), we obtain
\begin{equation}\label{ff1}
\begin{aligned}
\|\partial_t\hat u_2\|_{L^\infty(0,T;L^2(\Omega))}&=\|w\|_{L^\infty (0,T;L^2(\Omega))} \\
&\leq C\big( \|\partial_t\mathcal{R}_{{int},\theta^*,\eta^*}\|_{L^2(\Omega_T)}
+\|\partial_t\mathcal{R}_{{sb},\theta^*}\|_{H^1(0,T;L^2(\partial\Omega))} \\
&\qquad\ +\|\mathcal{R}_{{int},\theta^*,\eta^*}(\cdot,0)\|_{L^{2}(\Omega)}
+\|\mathcal{L}^*\mathcal{R}_{tb,\theta^*}\|_{L^{2}(\Omega)}\big).
\end{aligned}
\end{equation}
Combining  \eqref{ff}-\eqref{ff1} together and using the Sobolev embedding theorem, we get
\begin{equation*}\label{ff3}
\begin{aligned}
\mathcal{E}_{G,q}
&=\|q^*-q\|_{L^2(\Omega)}\\
&\leq C\big(\|\mathcal{R}_{{int},\theta^*,\eta^*}\|_{H^1(0,T;L^2(\Omega))}
+\|\mathcal{L}^*\mathcal{R}_{d,\theta^*}\|_{L^2(\Omega)}+\|\mathcal{L}^*\mathcal{R}_{{tb},\theta^*}\|_{L^{2}(\Omega)} \\
&\ \qquad+\|\partial_t\mathcal{R}_{{sb},\theta^*}\|_{H^1(0,T;L^2(\partial\Omega))}
+\epsilon(1+\delta\epsilon^{-3})+\epsilon+\delta\big)\\
&\leq C
\big(\|\mathcal{R}_{{int},\theta^*,\eta^*}\|_{H^1(0,T;L^2(\Omega))}
+\|q^*\mathcal{R}_{d,\theta^*}\|_{L^2(\Omega)}
+\|\Delta\mathcal{R}_{d,\theta^*}\|_{L^2(\Omega)}+\|q^*\mathcal{R}_{{tb},\theta^*}\|_{L^{2}(\Omega)}\\
&\qquad\
+\|\Delta\mathcal{R}_{{tb},\theta^*}\|_{L^{2}(\Omega)}
+\|\partial_t\mathcal{R}_{{sb},\theta^*}\|_{H^1(0,T;L^2(\partial\Omega))}
+\epsilon(1+\delta\epsilon^{-3})+\epsilon+\delta\big),
\end{aligned}
\end{equation*}
with $C=C(\|q^*\|_{L^\infty(\Omega)},\Omega,T)>0$. Picking $\epsilon=O(\delta^{1/3})$, we achieve the estimate
\begin{equation}\label{ff3}
\begin{aligned}
\mathcal{E}_{G,q}
&=\|q^*-q\|_{L^2(\Omega)}\\
&\leq C
\big( \|\mathcal{R}_{{int},\theta^*,\eta^*}\|_{H^1(0,T;L^2(\Omega))}
+\|q^*\mathcal{R}_{d,\theta^*}\|_{L^2(\Omega)}
+\|\Delta\mathcal{R}_{d,\theta^*}\|_{L^2(\Omega)}\\
&\qquad\ +\|q^*\mathcal{R}_{{tb},\theta^*}\|_{L^{2}(\Omega)}
+\|\Delta\mathcal{R}_{{tb},\theta^*}\|_{L^{2}(\Omega)}
+\|\partial_t\mathcal{R}_{{sb},\theta^*}\|_{H^1(0,T;L^2(\partial\Omega))}
+O(\delta^{1/3})\big).
\end{aligned}
\end{equation}
Finally, we will evaluate the generalization error for the unknown $u(x,t)$ employing the obtained generalization error \eqref{ff3} of the potential function. From the classical regularity theory for PDE, if $F(x,t)$ is sufficiently smooth, then it holds that  $u\in{L^2(0,T;L^\infty(\Omega))}$. Consequently,
\begin{equation*}\label{uu3}
\begin{aligned}
 \mathcal{E}_{G,u}
&=\|\hat u\|_{C([0,T];L^2(\Omega))} \\
&\leq C\big(\|\mathcal{R}_{{int},\theta^*,\eta^*}+(q^*-q)u[q]\|_{L^2(\Omega_T)}+\|\mathcal{R}_{{sb},\theta^*}\|_{H^1(0,T;L^2(\partial\Omega))}
+ \|\mathcal{R}_{{tb},\theta^*}\|_{L^{2}(\Omega)}\big)\\
&\leq
C\big(\|\mathcal{R}_{{int},\theta^*,\eta^*}\|_{L^2(\Omega_T)}
+\|q^*-q\|_{L^2(\Omega)}\;\|u\|_{L^2(0,T;L^\infty(\Omega))}+\|\mathcal{R}_{{sb},\theta^*}\|_{H^1(0,T;L^2(\partial\Omega))}\\
&\qquad\ +\|\mathcal{R}_{{tb},\theta^*}\|_{L^{2}(\Omega)}\big) \\
&\leq C
\big(\|\mathcal{R}_{{int},\theta^*,\eta^*}\|_{H^1(0,T;L^2(\Omega))}
+\|q^*\mathcal{R}_{d,\theta^*}\|_{L^2(\Omega)}
+\|\Delta\mathcal{R}_{d,\theta^*}\|_{L^2(\Omega)}+\|\mathcal{R}_{{tb},\theta^*}\|_{L^{2}(\Omega)}\\
&\qquad\ +\|q^*\mathcal{R}_{{tb},\theta^*}\|_{L^{2}(\Omega)}+
\|\Delta\mathcal{R}_{{tb},\theta^*}\|_{L^{2}(\Omega)}+\|\mathcal{R}_{{sb},\theta^*}\|_{H^2(0,T;L^2(\partial\Omega))}+O(\delta^{1/3})\big).
\end{aligned}
\end{equation*}
The proof is complete.
\end{proof}

The estimate \eqref{IPf} demonstrates that well-trained neural networks will produce small generalization errors for the inverse problem. Specifically, when all components of the training errors, including the interior PDE errors, measurement errors as well as the initial and boundary value ones, are sufficiently small, and the training sampling is large enough, the generalization errors for inverse problem using neural networks can be limited well.
This differs from classical stability results that rely solely on the knowledge of data. In this work, the generalization error estimates reflect stability due to both the model itself and the reconstruction algorithm. From Theorem \ref{main}, we see that we can limit the errors of both the inverse and forward problems by controlling the residuals and the mollified parameter. This provides important insights into the mathematical properties of our approach and plays an important role on the construction of algorithms.

\section{Numerical reconstructions.}\label{sec_num}
\subsection{Reconstruction algorithm.}

The neural networks $u_\theta(x,t)$ and $q_\eta(x)$ depend on the parameters $\theta$ and $\eta$
describing the networks information for specific activation functions.
Within the standard paradigm of deep learning,
one trains the networks by finding the optimal parameters $(\theta^*,\eta^*)$
such that the loss function (\ref{Loss1-1}) is minimized.
Our target is the unknown solution of the inverse problem
\eqref{eq1}-\eqref{eq3} and we wish to find the
trainable parameters $(\theta^*, \eta^*)$ such that the corresponding neural networks
$(u_{\theta^*}, q_{\eta^*})$ approximate $(u,q)$ well.
More precisely, to solve \eqref{eq1}-\eqref{eq3}
we first parameterize $u$ and $q$ by deep neural networks $u_\theta$ and $q_\eta$
with network parameters $(\theta,\eta)$ respectively.
Then, we design an appropriate loss function, which is minimized to determine the parameters $(\theta,\eta)$. Finally, a gradient-based method is applied to alternately update the network parameters so that $(u_\theta, q_\eta)$ gradually approximates $(u,q)$ for our inverse problem.

We provide a schematic of the neural networks in Figure \ref{Schematic}.
The left part visualizes two unknowns as two standard neural networks parameterized by $\theta$ and $\eta$, respectively. The right part applies the given physical laws to the networks. B.C., I.C. and D are the boundary condition, initial status and the measurement data obtained from random sample points in training sets $\mathcal{S}_{sb}$, $\mathcal{S}_{tb}$ and $\mathcal{S}_{d}$ respectively. The training points in $\mathcal{S}_{int}$ are randomly sampled as the PDE residuals points in interior spatio-temporal domain. The loss function with some Sobolev norm is computed on the sample points, which can be done efficiently through automatic differentiation (AD) in case of derivative information. Minimizing the loss with respect to the parameters $(\theta, \eta)$ alternately produces $u_{\theta^\ast}$ and $q_{\eta^\ast}$, which serves as the  approximation to the solution of the inverse problem.

 \begin{figure}[h!]
 \centering
 \includegraphics[width=0.8\textwidth,height=0.2\textheight]{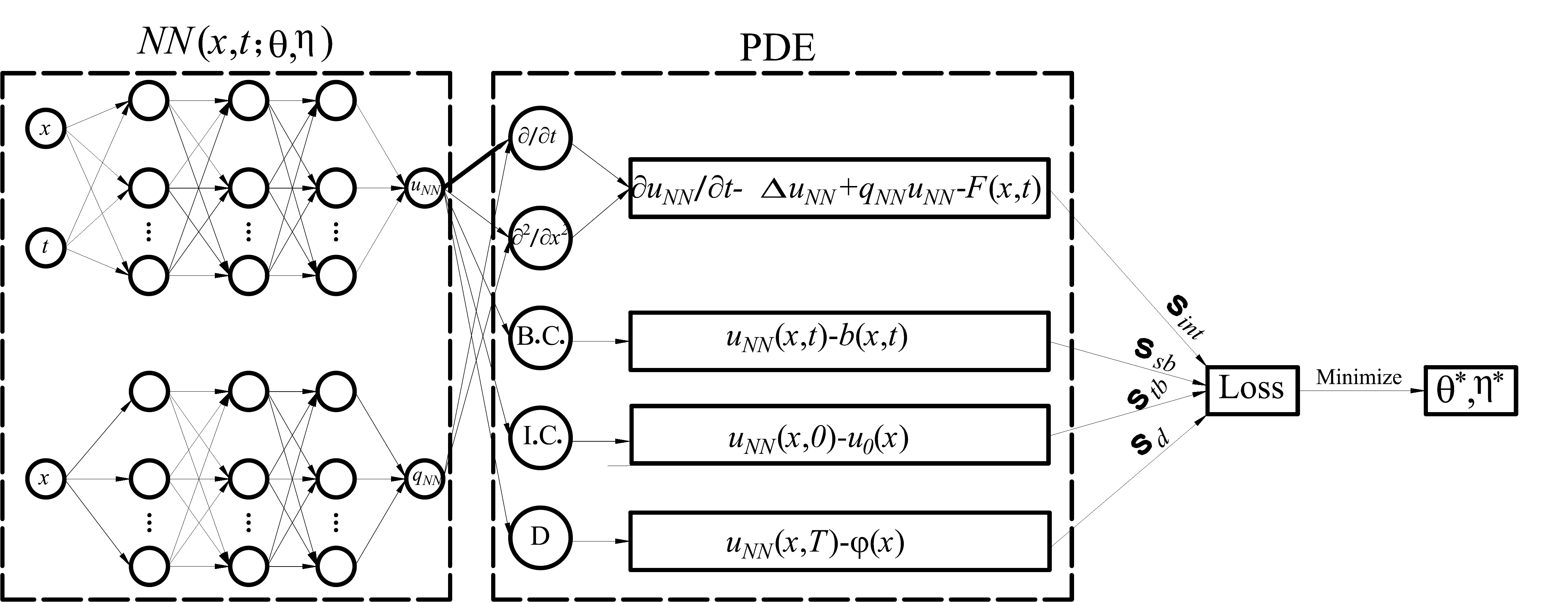}
 \caption
 {The schematic of the deep neural networks for solving the inverse potential problem.}\label{Schematic}
 \end{figure}

With the support of Theorem \ref{main}, we can construct the proposed Algorithm \ref{alg1} for solving the inverse problem \eqref{eq1}-\eqref{eq3}.

\begin{algorithm}[h!]
\caption{
  Data-driven solution of the inverse problem (\ref{eq1})-(\ref{eq3}).
}
\label{alg1}
\begin{algorithmic}
\REQUIRE
 Noisy data $\varphi^\delta$ for the inverse problem.

\STATE {\bf Initialize}
Network architectures $\left(q_{\eta},u_{\theta}\right)$ and parameters $(\eta,\theta)$.
\FOR {$j=1,\cdots,K$}
 \STATE
Sample
$\mathcal{S}_{d},\ \mathcal{S}_{int},\ \mathcal{S}_{sb},\ \mathcal{S}_{tb}$.
\STATE
$\eta\leftarrow \operatorname{Adam}\left(-\nabla_{\eta} J_{\lambda}^N(\eta, \theta), \tau_{\eta}, \lambda_{\eta}\right)$,\\
$\theta  \leftarrow \operatorname{Adam}\left(-\nabla_{\theta} J_{\lambda}^N(\eta, \theta), \tau_{\theta}, \lambda_{\theta}\right),$

\ENDFOR
\ENSURE
$\left(q_{\eta},u_{\theta}\right)$.
\end{algorithmic}
\end{algorithm}

 The above minimization problem is to search a minimizer of a possibly non-convex function
 $J_\lambda^N(\theta,\eta)$ over $\Theta\subset \mathbb{R}^\mathcal{M}$ for possibly very large $\mathcal{M}$.
 The hyper-parameters $(\tau_{\eta},\tau_{\theta})$ are learning rates and $(\lambda_{\eta},\lambda_{\theta})$
 are balance hyper-parameters between PDE and measurement data residuals.
The robust analysis for hyper-parameters $\lambda$ and the architectures of neural networks are studied in the next subsection.

The optimizer in Algorithm \ref{alg1} is Adam (Adaptive Moment Estimation), which is an optimization algorithm commonly used in deep learning for training neural networks. The key idea of Adam is to adaptively adjust the learning rate for each parameter based on estimates of both the first-order moment (the mean) and the second-order moment (the uncentered variance) of the gradients. This adaptation helps Adam to perform well in different types of optimization problems. The algorithm maintains an exponentially moving average of gradients ($m_t$) and squared gradients ($V_{t}$) for each parameter. At each iteration, Adam updates the parameters using a combination of these moving average estimates. It incorporates bias correction to account for the fact that the estimates are biased towards zero at the beginning of training.
We set $g_{t}$ be gradients w.r.t. stochastic objective at timestep $t$, $\beta_1, \beta_2\in[0,1)$ be the exponential decay rates for the moment estimates and $\tau$ be the initial learning rate. Good default settings for the tested machine learning problems are $\tau=0.001$, $\beta_1=0.9$, $\beta_2=0.999$ and fuzzy factor $\epsilon=10^{-8}$.
The updates are calculated as follows:
\begin{itemize}
 \item[(1)] Initialize the first moment vector $m$ and the second moment vector $V$ with zeros for each parameter:
$$m_{0}=V_{0}=0.$$
\item[(2)] Update the first moment estimate $m$ using a weighted average of the current gradient $g_{t}$ and the previous first moment estimate $m_{t-1}$:
$$ m_t=\beta_1m_{t-1}+(1-\beta_1)g_{t}.$$
\item[(3)] Update the second moment estimate $V$ using a weighted average of the squared gradients and the previous second moment estimate $V_{t-1}$:
$$V_{t}=\beta_{2}V_{t-1}+(1-\beta_{2})g_{t}^{2}.$$
\item[(4)] Calculate the bias-corrected first and second moment estimate
to correct for their initialization bias:
$${\hat{m}}_{t}={\frac{m_{t}}{1-(\beta_{1})^{t}}}, \quad
{\hat{V}}_{t}={\frac{V_{t}}{1-(\beta_{2})^{t}}}. $$
\item[(5)] Update the parameters $\xi$ by moving in the direction of the first moment estimate, where the learning rate is $\tau$ divided by the square root of the second moment estimate:
$$\xi_{t}=\xi_{t-1}-\tau\frac{\hat{m}_{t}}{\sqrt{\hat{V}_{t}}+\epsilon}.$$
\end{itemize}

The hyper-parameters in Adam include the learning rate and the decay rates for the moving averages. These hyper-parameters need to be tuned based on the specific problem and dataset to achieve optimal performance. Adam has several advantages that make it popular in deep learning:
\begin{itemize}
 \item [(a)] Adaptive learning rate: Adam automatically adapts the learning rate for each parameter based on the estimated first and second moments. This adaptive behavior helps in effectively navigating the optimization landscape and can lead to faster convergence.
 \item [(b)] Efficiency: Adam uses the moving averages to maintain a history of gradients, which eliminates the need to store and compute gradients for each iteration separately. This makes Adam memory-efficient and allows for efficient parallelization during training.
 \item [(c)] Robustness: Adam performs well across a wide range of optimization problems and is less sensitive to hyper-parameter tuning compared to some other optimizers. It can handle sparse gradients and noisy data effectively.
\end{itemize}

The proposed algorithm, which utilizes (\ref{Loss1}) as the loss function, exhibits superior performance in recovering smooth solutions due to the high regularity of the PDEs residual term. This regularity term promotes smoother solutions and is an important factor in achieving higher accuracy.
Furthermore, the use of automatic differentiation (AD) implementations enables the efficient calculation of the necessary derivatives. This feature is a significant advantage of our approach as it allows for the accurate optimization of the objective function, which is crucial for effective solution of inverse problems. To validate the effectiveness of the proposed algorithm and to substantiate our claims, we conduct a series of numerical experiments.

\subsection{Numerical experiments.}

In this subsection, we will present several numerical examples for the spatial domain $\Omega\subset \mathbb{R}^d$ with $d=2,3$.
We define the following relative errors for exact solutions $(u,q)$ and numerical approximations $(u^*,q^*)$ as
\begin{eqnarray*}
{Re}_u:=\frac{\|u-u^*\|_{L^2(\Omega_T)}}{\|u\|_{L^2(\Omega_T)}},\quad
{Re}_q:=\frac{\|q-q^*\|_{L^2(\Omega)}}{\|q\|_{L^2(\Omega)}},\quad
{Re}_{\Delta u}:=\frac{\|\Delta u-\Delta u^*\|_{L^2(\Omega_T)}}{\|\Delta u\|_{L^2(\Omega_T)}}.
\end{eqnarray*}

\paragraph{Example 1 (two-dimensional experiment):}
For equation \eqref{eq1}, we set the exact solution $u$ and the domain $\Omega_T$ as
$$ u(x,y,t)=(x^2+y^2+1)\exp(t), \quad (t,x,y)\in\Omega_T=[0,1]^3.$$
The exact potential $q$ is given as
\begin{equation*}
 q(x,y)=\sin(\pi x)\sin(\pi y).
 \end{equation*}
The initial and boundary conditions can be calculated from the representation of $u$ straightforwardly. The exact measurement will be
\begin{eqnarray*}\label{SS2}
u(x,y,1)=\varphi(x,y)=(x^2+y^2+1)\exp(1),
\end{eqnarray*}
and in our experiments the noisy data is set as
\begin{eqnarray}\label{random}
\varphi^\delta(x,y):=\varphi(x,y)+\delta\cdot (2\; \text{rand}(\text{shape}(\varphi(x,y)))-1),
\end{eqnarray}
where $\text{rand}(\text{shape}(\varphi))$ is a random variable generated by uniform distribution in $[0,1]$.

For the implementation details, we use a fully connected neural network for $u_{\theta}$ and $q_{\eta}$ with 3 hidden layers, each with a width of 20.
We take $$\text{N}=N_{int}+N_{sb}+N_{tb}+N_{d}=256+256\times4+256+256=1792$$ as the number of collocation points, which are randomly sampled in four different domains, i.e., interior spatio-temporal domain, spatial and temporal boundary domain, and additional measurement domain. The activation function is $\tanh(x)$ and the hyper-parameter is $\lambda=0.01$. The number of training epochs is set to be $5\times10^4$, and the initial learning rates $\tau_\theta, \tau_\eta$ both start with $0.001$ and shrink $10$ times every $2\times10^4$ iteration.
The test sets are chosen by a uniform mesh
\begin{eqnarray}\label{testset}
\mathcal{T}:=\{(t_k,x_i,y_j):  k,i,j=0,1,\cdots,49\}\subset \Omega_T.
\end{eqnarray}
Since the noisy level of the measurement data affects the reconstruction accuracy, in this simulation, we test the training performance for various noisy levels. Figure \ref{PINN_Training-2D-1} records the training process, i.e., the training loss, the relative error for the reconstruction of $q$ and the relative error for the recovery of $u$ with respect to the iterations for different noise levels $\delta=0.1\%,\;1\%,\;5\%,\;10\%$ by the proposed scheme.
After training, we test the reconstruction result on test sets $\mathcal{T}$.
The distribution of the temperature field $u(x,t)$ also depends on the time $t$, Figure \ref{PINN-time-2D-1} shows the time-series relative error of the recovered $u$ with various noise levels after logarithmic re-scaling. As shown in these figures, the training performance deteriorates as the noise level of the measurement data increasing.

\begin{figure}[h!]
\centering
\includegraphics[width=1.1\textwidth,height=0.2\textheight]{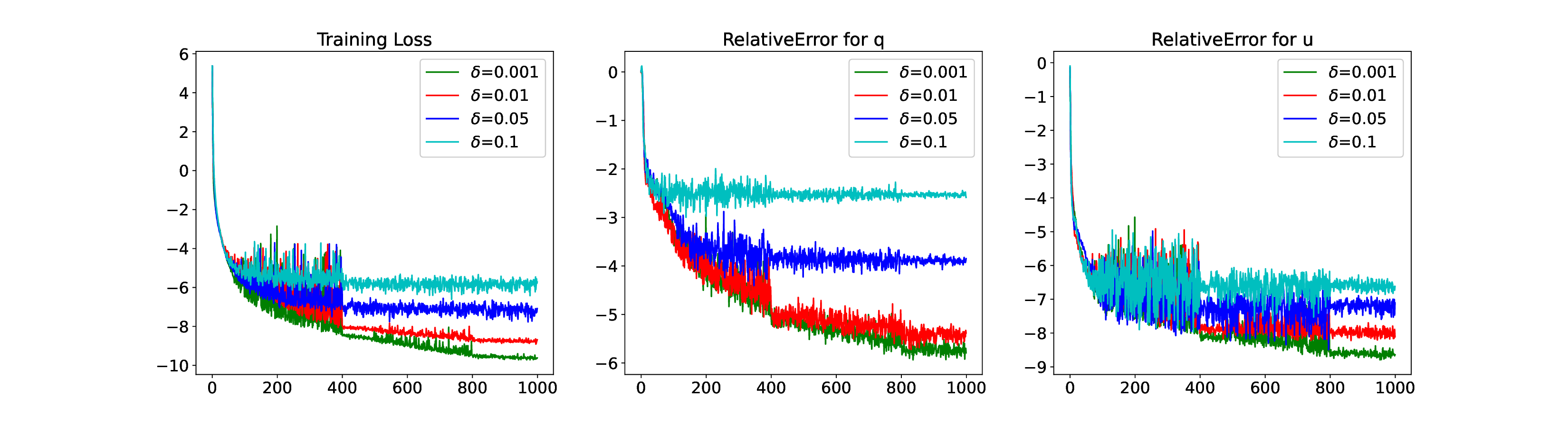}
\caption
{The training loss (left), the relative error for $q$ (middle), the relative error for $u$ (right) after logarithmic re-scaling.
}\label{PINN_Training-2D-1}
\end{figure}

\begin{figure}[h!]
\centering
\includegraphics[width=0.4\textwidth,height=0.2\textheight]{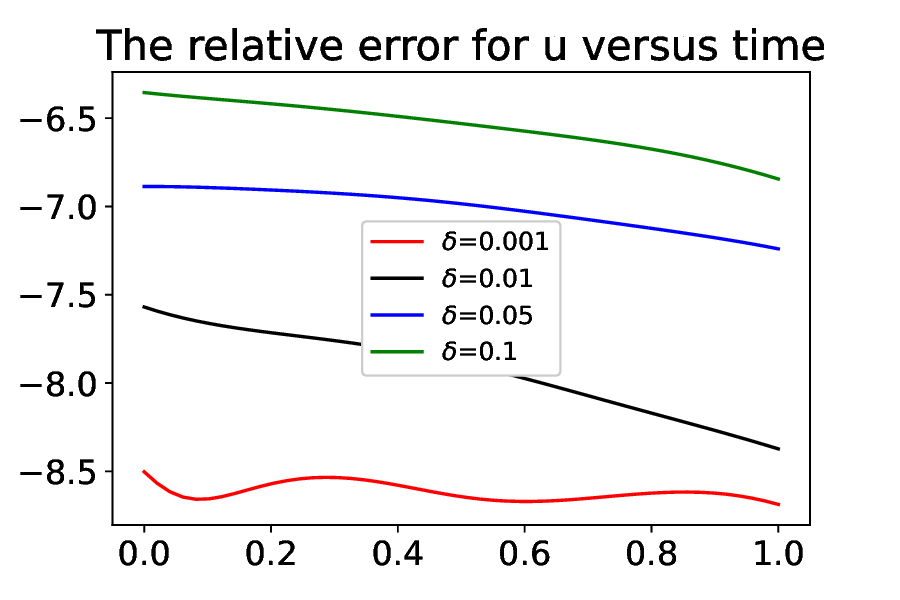}
\caption{The time series relative error (test) for $u$ for different noisy data.}\label{PINN-time-2D-1}
\end{figure}

Figure \ref{PINN-q-exact-2D-1} shows the exact solution of the ptential term. Figure \ref{PINN-q-2D-1} shows the reconstruction results for $q(x)$ by optimizing the proposed loss function (first line) and the corresponding absolute pointwise error for various noisy level $\delta=0.1\%,5\%,10\%$ (second line).
Meanwhile, Figure \ref{PINN-u-2D-1} presents the reconstruction solution $u$ (first line) and corresponding absolute pointwise error (second line) for various noisy level measurement data at $t=1/7$. We can see that the reconstruction accuracy for $q$ deteriorates as the noise level of the measurement data increasing, but the performance for $u$ is still satisfactory.

\begin{figure}[h!]
\centering
\includegraphics[width=0.4\textwidth,height=0.2\textheight]{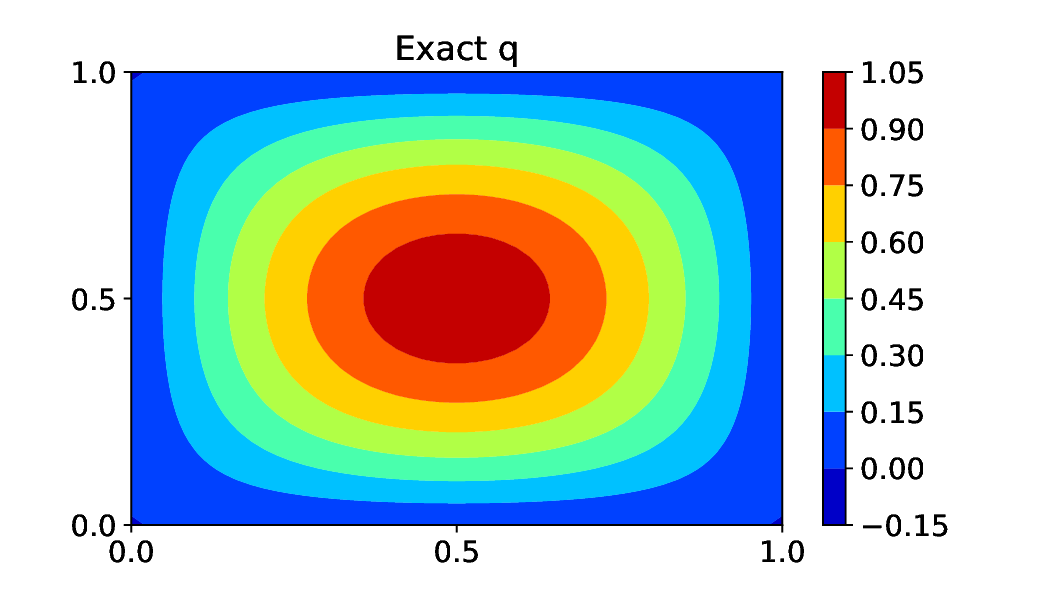}
\caption
{The exact potential function $q$.}\label{PINN-q-exact-2D-1}
\end{figure}

\begin{figure}[h!]
\centering
\includegraphics[width=1.1\textwidth,height=0.35\textheight,center]{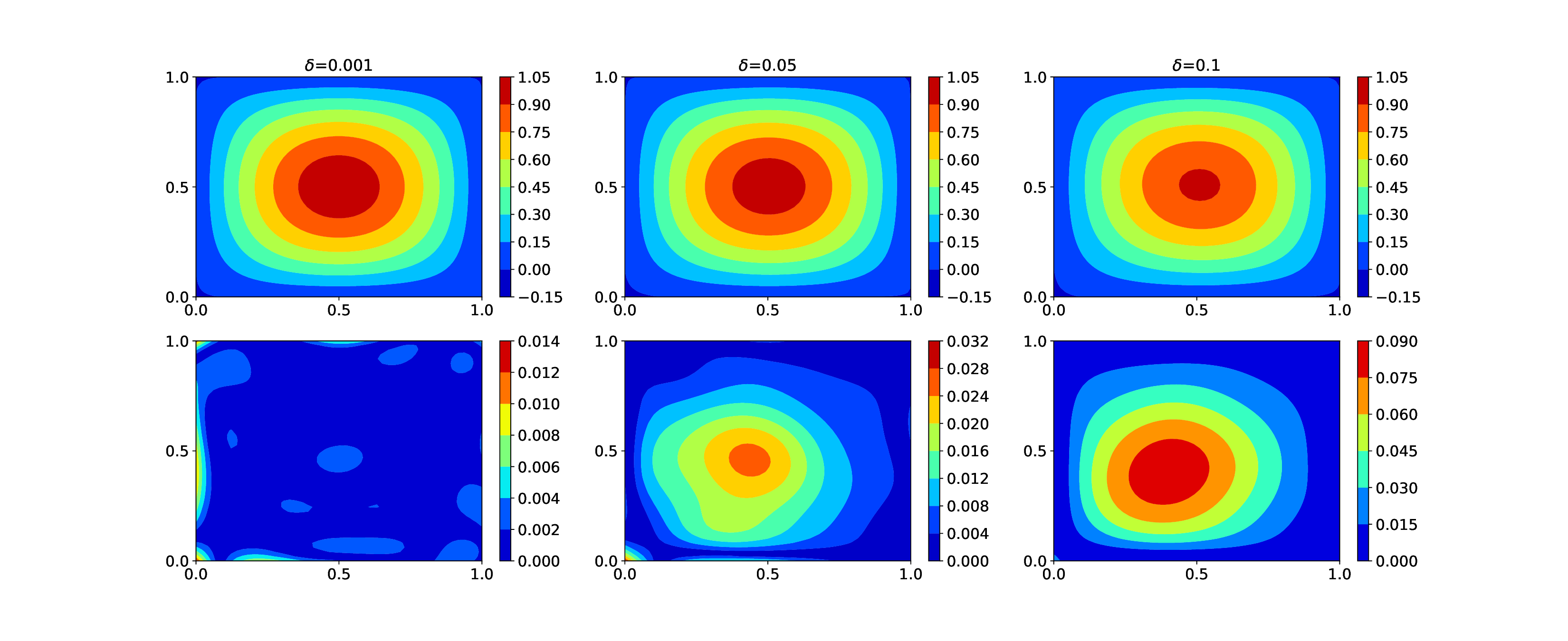}
\caption{The reconstruction of potential function $q$ (upper) and corresponding absolute pointwise error $|q-q^*|$ (bottom) for various noisy level measurement data.}\label{PINN-q-2D-1}
\end{figure}

\begin{figure}[h!]
\centering
\includegraphics[width=1.1\textwidth,height=0.35\textheight,center]{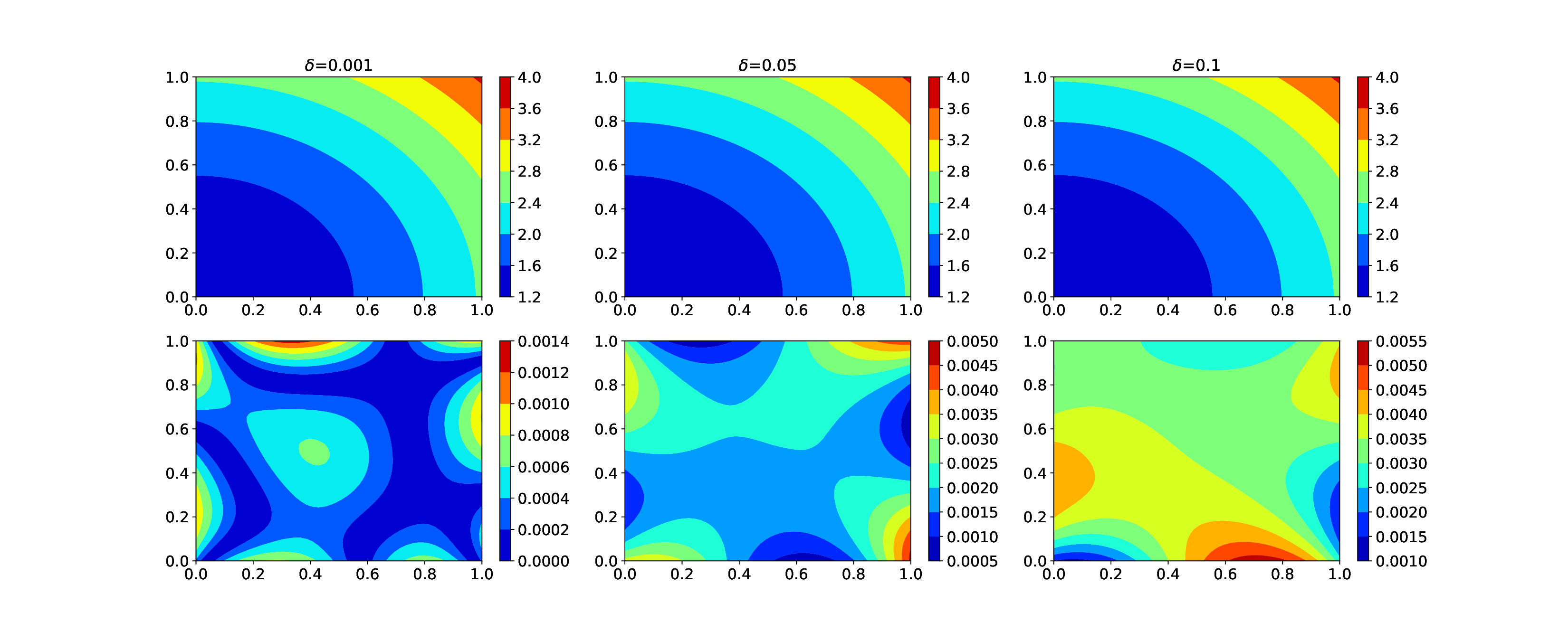}
\caption{The reconstruction of solution $u$ (upper) and corresponding absolute pointwise error $|u-u^*|$ (bottom) for various noisy level measurement data at $t=1/7$.}\label{PINN-u-2D-1}
\end{figure}

Table \ref{ex1-table} presents the recovery results solved by two schemes:
(I) the proposed frameworks with the loss function (\ref{Loss1}),
(II) the DGM frameworks with the loss function (\ref{Loss-general}).
We record the generalization error of $q$, $u$ and $\Delta u$ in $L^2$-error from the noisy input data with $\delta=0.01$. Due to the random sampling of the training data points, the inversion results have some stochasticity.
Thus, we perform Algorithm \ref{alg1} with the loss function in the formulation (\ref{Loss1}) and formulation (\ref{Loss-general}) five times, respectively.
The relative errors (mean and standard deviation) for the recovery of $q$, $u$ and $\Delta u$ are shown in Table \ref{ex1-table}. As observed, optimizing the loss function proposed in this paper leads to more accurate recovery results, especially for the reconstruction of $q$ compared with DGM frameworks. Moreover, although the reconstruction accuracy of $u$ in $L^2$-error for both two frameworks are relatively close, the accuracy of $\Delta u$ in $L^2$-error for proposed scheme in this paper performs better. This suggests that the proposed frameworks are better able to capture smooth solutions.

\begin{table}[h!]
\centering
\caption{
  The inversion results with noisy measuremrnt data ($\delta=0.01$) solved by two schemes: (I) the proposed frameworks with the loss function (\ref{Loss1}),
(II) the DGM frameworks with the loss function (\ref{Loss-general}).}
\label{ex1-table}
\tabcolsep0.04 in
\begin{tabular}{ccccccccc}
\specialrule{0.05em}{3pt}{3pt}
     &$Re_q$
 &$Re_u$ &$Re_{\Delta u}$ \\
 \specialrule{0.05em}{3pt}{3pt}
  $\text{I}$
  &$0.5011\% \pm 0.0102\%$  &$0.0215\% \pm 0.0049\%$
   & $0.2550\% \pm 0.0169\%$\\
\specialrule{0.05em}{3pt}{3pt}
  $\text{II}$

  &$12.4545\% \pm 4.8016\%$  &$0.1148\% \pm 0.0133  \%$   &$2.3468\% \pm 0.8133\%$\\
\specialrule{0.05em}{3pt}{3pt}
\end{tabular}
\end{table}

\paragraph{Example 2 (two-dimensional experiment):}
For equation \eqref{eq1}, we set the exact solution $u$ and the domain $\Omega_T$ as
$$ u(x,y,t)=t\exp(x+y), \quad (t,x,y)\in\Omega_T=[0,1]\times[0,2]^2.$$
The exact potential $q$ is given as
\begin{equation*}
 \begin{aligned}
  q(r)&=
\begin{cases}
15\left(\cos r-\sqrt{3}/2\right)+2, \quad 0\leq r \leq \pi/6, \\
2, \quad \text { otherwise },
\end{cases}\\
r(x,y)&=\sqrt{(x-1)^2+(y-1)^2}.
 \end{aligned}
\end{equation*}
The exact measurement will be
\begin{eqnarray*}\label{SS2}
u(x,y,1)=\varphi(x,y)=\exp(x+y),
\end{eqnarray*}
and the noisy data $\varphi^\delta$ is generated by \eqref{random}.
The network architectures and hyper-parameters such as activation function, balance hyper-parameter $\lambda$ are all the same as Example 1.
The number of training epochs is set to be $1\times10^5$, and the initial learning rates $\tau_\theta, \tau_\eta$ both start with $0.001$ and shrink $10$ times every $2\times10^4$ iteration. The test sets are chosen by a uniform mesh as \eqref{testset}

In this simulation, we evaluate the training performance under various levels of measurement noise. The training process under Algorithm \ref{alg1} is recorded in Figure \ref{PINN_Training-2D}, which includes the training loss, and the relative errors for the reconstruction of $q$ and $u$ during training process for different noise levels ($\delta=0,\ 1\%,\ 5\%,\ 10\%$). Figure \ref{PINN-q-exact-2D} displays the exact potential function $q$, while the approximated $q^*$ under different noise level measurements ($\delta=1\%,\ 5\%,\ 10\%$) are shown in Figure \ref{PINN-q-2D}. We can see that the numerical reconstructions still satisfy the theoretical results even with zero initial condition and nonsmooth exact potential. This means that in numerical reconstructions we may release the conditions of Assumption \ref{assumption} to some extent.

\begin{figure}[h!]
\centering
\includegraphics[width=1.1\textwidth,height=0.2\textheight]{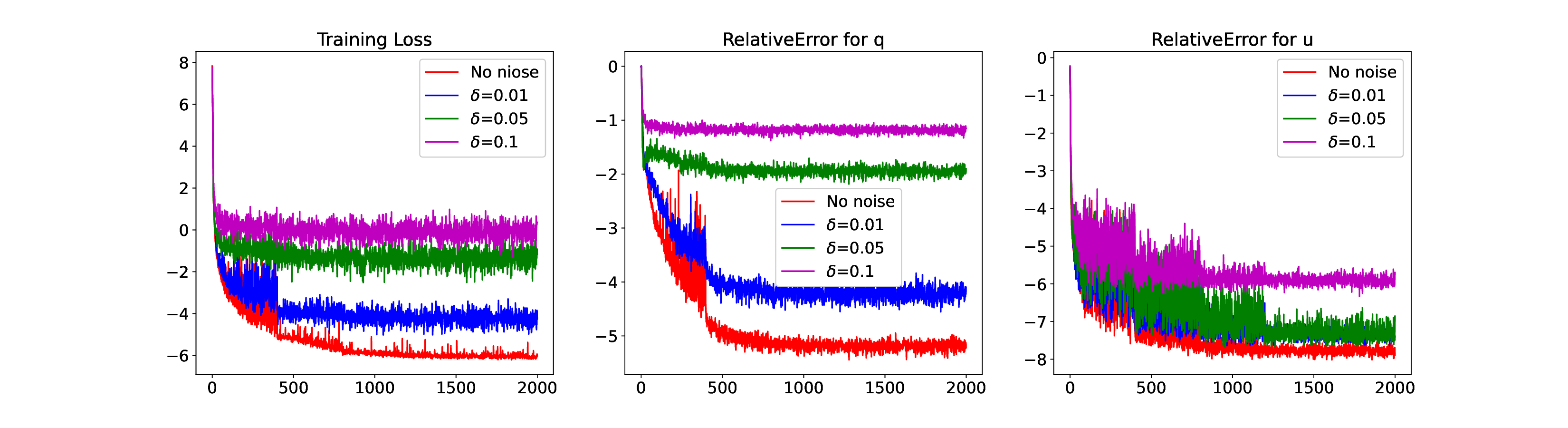}
\caption
{The training loss (left), the relative error for $q$ (middle), the relative error for $u$ (right) after logarithmic re-scaling.
}\label{PINN_Training-2D}
\end{figure}

\begin{figure}[h!]
\centering
\includegraphics[width=0.4\textwidth,height=0.2\textheight]{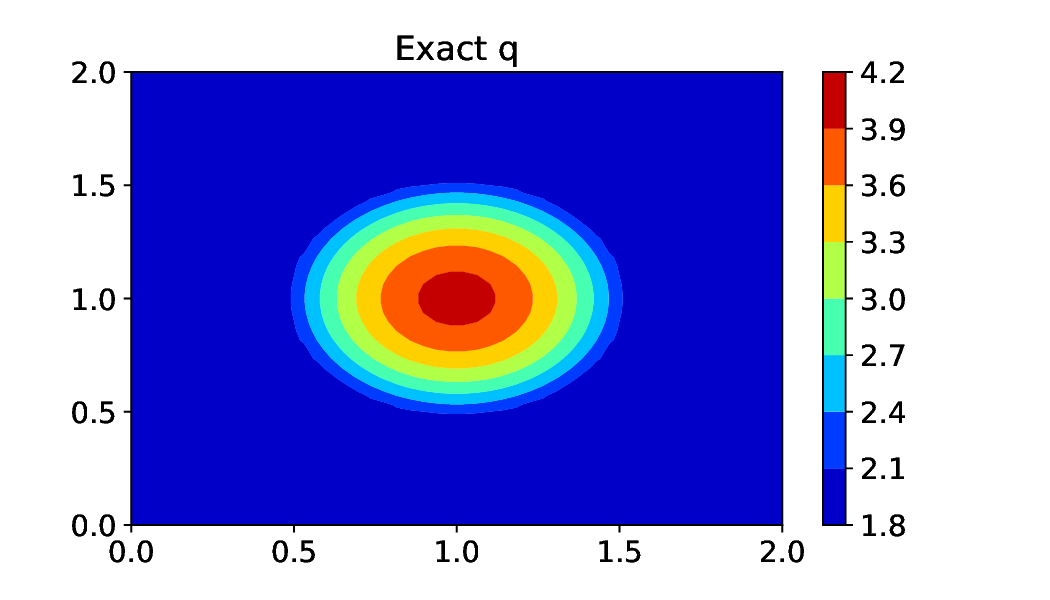}
\caption
{The exact potential function $q$.}\label{PINN-q-exact-2D}
\end{figure}

\begin{figure}[h!]
\centering
\includegraphics[width=1.1\textwidth,height=0.35\textheight,center]{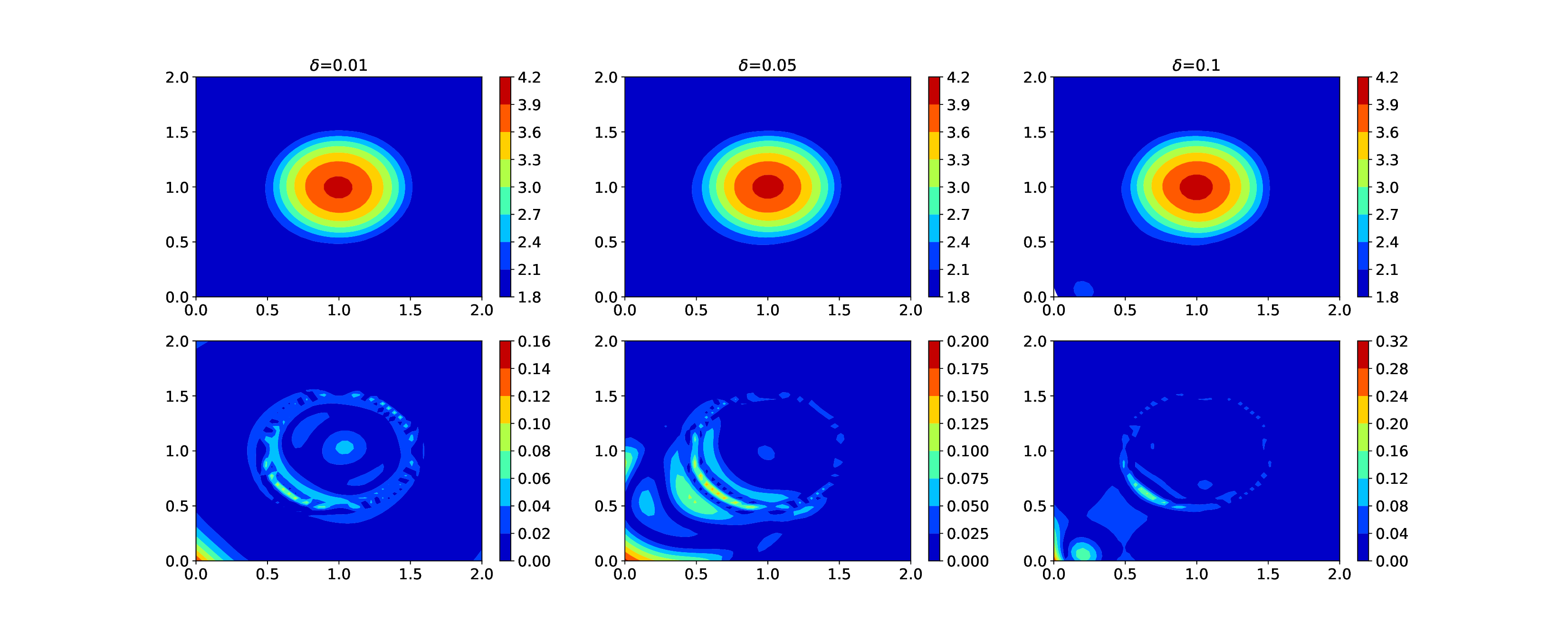}
\caption{The reconstruction of potential function $q$ (upper) and corresponding absolute pointwise error $|q-q^*|$ (bottom) for various noisy level measurement data.}\label{PINN-q-2D}
\end{figure}

Now, we start to verify the convergence of the iteration in Theorem \ref{main} with different neural network architectures. In the experiments, for a fixed number of per-layer neurons  $\text{NN}=20$, we compute the reconstruction errors for $q$ versus the noise level $\delta$ using logarithmic re-scaling with various hidden layers $\text{NL}=3, 4, 6$. The results of these experiments are presented in the left of Figure \ref{ErrorRate}. The theoretical estimate
$O(\delta^{1/3})$ is shown by the black line. Similarly, fixing hidden layer $\text{NL}=6$, the reconstruction errors for $q$ under various per-layer neurons $\text{NN}=10,15,20$ are given in the right of Figure \ref{ErrorRate}. From this figure, we see that the error could be bounded by the rate $\delta^{1/3}$ to some extent, which supports the theoretical analysis in Theorem \ref{main}.

\begin{figure}[h!]
\centering
\includegraphics[width=1.0\textwidth,height=0.25\textheight,center]{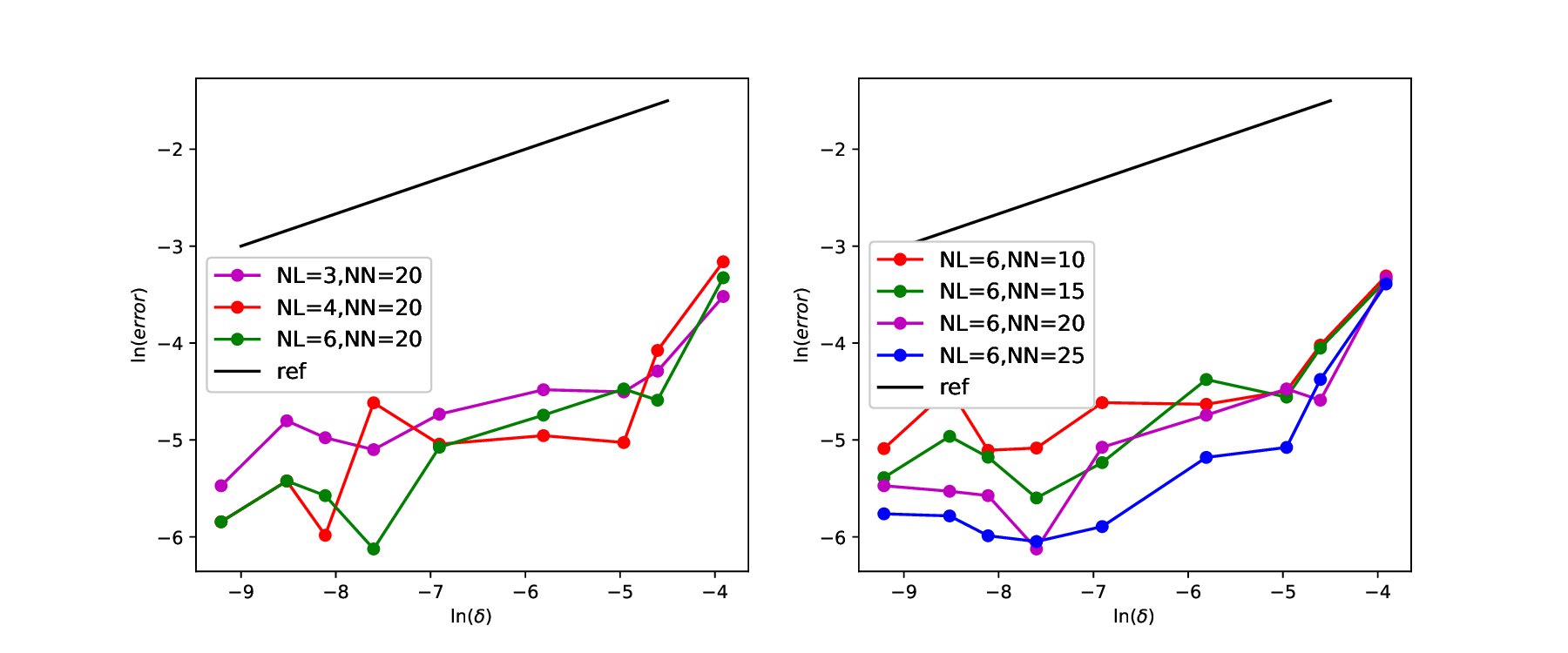}
\caption{The reconstruction error of $q$ versus noise level $\delta$ under various layer numbers for fixed per-layer neuron numbers $\text{NN}=20$ (left), and under various per-layer neuron for fixed layer numbers $\text{NL}=6$ (right). }\label{ErrorRate}
\end{figure}

In order to evaluate the effectiveness of the proposed scheme in terms of hyper-parameters and network structure, a series of experiments are conducted. Specifically, we examine the impact of the balance hyper-parameter $\lambda$ in \eqref{Loss1} and the network structure, including the number of hidden layers and neurons. For a fixed number of hidden layers ($\text{NL}=3$) and a fixed number of neurons per-layer ($\text{NN}=20$), we compute the reconstruction errors (mean and standard deviation) for $q$ and $u$ using various values of $\lambda$, such as
$\lambda=10^{j},\ -4\le j\le 1$. The results of these experiments are presented in Table \ref{tab4-D-Ab1-0}, which indicates that the performance of the inverse problem is highly dependent on the balance hyper-parameter $\lambda$. Specifically, we find that the relative reconstruction errors are optimized when $\lambda$ is set to $10^{-2}$. Furthermore, we observe that the reconstruction errors increase significantly as $\lambda$ exceeds this optimal value. These results suggest that the selection of the balance hyper-parameter is critical to achieving good performance in this inverse problem.

Next we experiment with various combinations of hidden layers and neuron numbers for the inverse problem using Algorithm \ref{alg1}. We set the dimension to $d=2$ and try a total of $16$ combinations of hidden layers (NL) and per-layer neuron numbers (NN), with
$$\text{NL}=3,6,9,14,\quad \text{NN}=10,15,20,25.$$ For each combination, we run Algorithm \ref{alg1} for $1\times 10^5$ iterations and record the relative errors (mean and standard deviation) for $(q,u)$ in Table \ref{tab4-D-Ab1}.
It indicates that deeper (larger NL) and/or wider (larger NN) neural networks tend to yield lower reconstruction errors, although this causes higher computational cost. However, we also observe that for fixed neuron number NN, increasing the number of hidden layers NL, for example $\text{NL}\geq 15$, causes the algorithm fail to converge as the number of iterations increases. This suggests that increasing the number of layers and/or neurons can enhance the representation capacity of neural networks. But it may also introduce more parameters to train, and lead to longer training times and potential overfitting of the representation.

\begin{table}[h!]
\centering
\caption{The relative errors for $q$ (first line) and $u$ (second line) using various hyper-parameter $\lambda$ from noisy measurement data with $\delta=0.01$, where the hidden layers $\text{NL}=3$, per-layer neurons $\text{NN}=20$. } \label{tab4-D-Ab1-0}
\tabcolsep0.04 in
\begin{tabular}{cccccccccc}
\specialrule{0.05em}{13pt}{3pt}
   Error &$\lambda=10^{-4}$  &$\lambda=10^{-3}$  &$\lambda=10^{-2}$ \\
\specialrule{0.05em}{3pt}{3pt}
  ${Re}_q$   &$2.5113\%\pm 1.1339\%$  & $0.9942\% \pm 0.1121\%$   & $0.7124\% \pm 0.0639\%$   \\
  \specialrule{0.05em}{3pt}{3pt}
  ${Re}_u$    &$ 0.1374\% \pm 0.0652\%$  & $0.1216\% \pm 0.0320\%$     & $0.0560\% \pm 0.0204\%$  \\
\specialrule{0.05em}{3pt}{3pt}
 Error   & $\lambda=10^{-1}$  &$\lambda=1$ &$\lambda=10$\\
\specialrule{0.05em}{3pt}{3pt}
  ${Re}_q$    &$0.8388\%\pm \underline{}0.2448\%$
 &$1.0809\% \pm 0.1741$\%   &$2.2891\%\pm 1.3302\%$\\
  \specialrule{0.05em}{3pt}{3pt}
  ${Re}_u$  &$0.0331\%\pm 0.0094\%$   &$0.0478\% \pm 0.0122\%$   &$0.0874\% \pm 0.0110\%$\\
\specialrule{0.05em}{3pt}{3pt}
\end{tabular}
\end{table}













\begin{table}[h!]
\centering
\caption{The relative errors for $q$ (top) and $u$ (bottom) using various combinations of (NL,NN) for problem dimension $d=2$ from noisy measurement data with $\delta=0.01$, where the hyper-parameter $\lambda=10^{-2}$. } \label{tab4-D-Ab1}
\tabcolsep0.04 in
\begin{tabular}{ccccccccc}
\specialrule{0.05em}{13pt}{3pt}
   $NL$ &$NN=10$  &$NN=15$ &$NN=20$ &$NN=25$ \\
\specialrule{0.05em}{3pt}{3pt}
  $3$   &$1.7258\% \pm 0.6258\% $  &$1.1664\% \pm 0.3531\%$  & $0.7124\% \pm 0.0639\%$   & $0.5978\% \pm 0.1071\% $\\

  $6$    &$ 1.5015\% \pm 0.6640\%$ & $0.7234  \% \pm 0.2407\% $    & $0.4482\% \pm 0.1065\% $  & $0.3093\% \pm 0.0310\%$\\

$9$    &$0.7797\% \pm 0.1577\%$   &$ 0.4788\% \pm 0.1598\% $   & $0.4796\% \pm 0.0900 \%$ & $0.4874\%\pm 0.2034\% $\\

$14$    &$0.5602\% \pm 0.2118\% $ & $0.4077\% \pm 0.1990 \%$   & $0.3958\% \pm 0.0552\%$  & $ 0.4108\% \pm 0.0637\%$\\

\specialrule{0.05em}{3pt}{3pt}
 $3$   &$0.0841 \% \pm 0.0216\%$  &$0.0615 \% \pm 0.0290\%$  & $0.0560\% \pm 0.0204\%$   & $0.0306\% \pm 0.0105\% $\\

  $6$    &$ 0.0575\% \pm 0.0238\%$ & $0.0490\% \pm 0.0062\% $    & $0.0280\% \pm 0.0071\% $  & $0.0358\% \pm 0.0090\%$\\

$9$    &$0.0751 \% \pm 0.0129\% $ &$ 0.0433\% \pm 0.0037\% $   & $0.0523\% \pm 0.0166 \%$ & $0.0314\% \pm 0.0132\% $\\

$14$    &$0.0806\% \pm 0.0103\%$  & $0.0612\% \pm 0.0144\% $   & $0.0593\% \pm 0.0145\%$  & $0.0422\% \pm 0.0115\%$\\

\specialrule{0.05em}{3pt}{3pt}
\end{tabular}
\end{table}

\paragraph{Example 3 (three-dimensional experiment):}
We also take the following 3-dimensional experiment. We set the exact solution and the domain of equation \eqref{eq1} as
\begin{eqnarray*}
u(x,y,t)=t\exp(x+y+z), \quad (t,x,y,z)\in\Omega_T=[0,1]^4.
\end{eqnarray*}
The exact potential $q$ is given as
\begin{eqnarray*}
q(x,y,z)=x+y+z.
\end{eqnarray*}

We also employ fully-connected neural networks with $\text{NL}=4,\ \text{NN}=20$ for both $u_{\theta}$ and $q_{\eta}$. The number of training points are $$\text{N}=N_{int}+N_{sb}+N_{tb}+N_{d}=256+256\times6+256+256=2304,$$ which are randomly sampled from four different training sets.
The other network architectures and hyper-parameters such as activation function, balance hyper-parameter $\lambda$, number of training epochs, and initial learning rate are all the same as Example 1.
The test sets are chosen by a uniform mesh
$$\mathcal{T}:=\{(t_k,x_i,y_j,z_l):\  k,i,j,l=0,1,\cdots,49\}\subset \Omega_T.$$
Figure \ref{PINN-q-exact-3D} shows the exact potential function $q$ and the relative errors versus iterations during training process for different noise scales $\delta=0,\ 1\%,\ 5\%,\ 10\%$. Figure \ref{PINN-q-3D} presents the potential functions $q_{\eta}$ recovered from different noise levels and the corresponding point by point absolute errors on test sets. The inversion results are satisfactory and reasonable overall.

\begin{figure}[h!]
\centering
\includegraphics[width=0.9\textwidth,height=0.2\textheight]{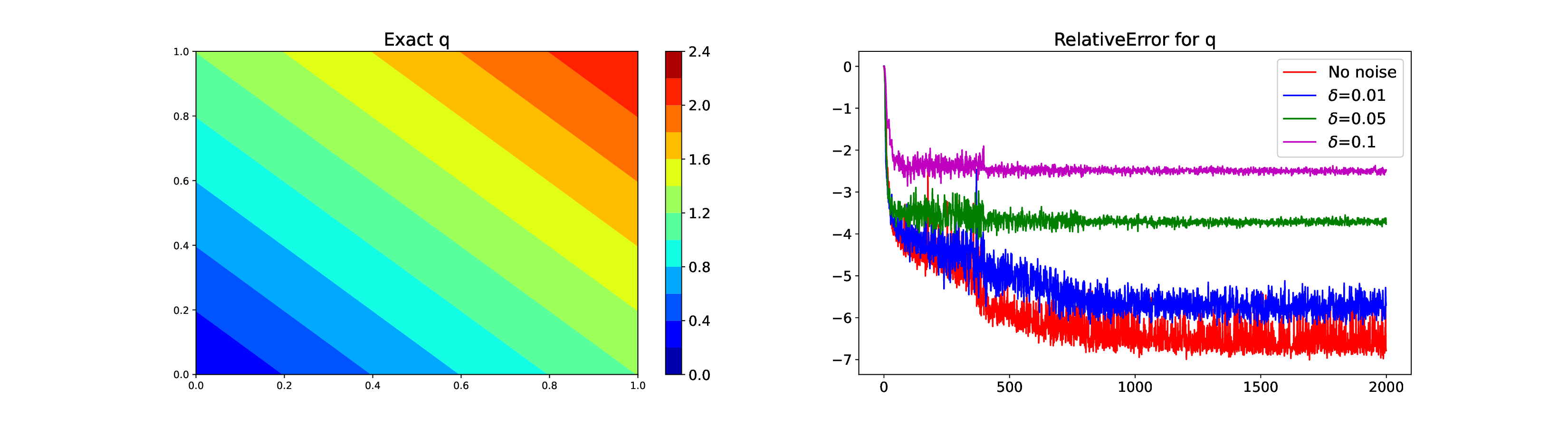}
\caption
{The exact potential function $q$ (left), the relative error for recovered $q$ with different noise levels (right).}\label{PINN-q-exact-3D}
\end{figure}

\begin{figure}[h!]
\centering
\includegraphics[width=1.1\textwidth,height=0.35\textheight,center]{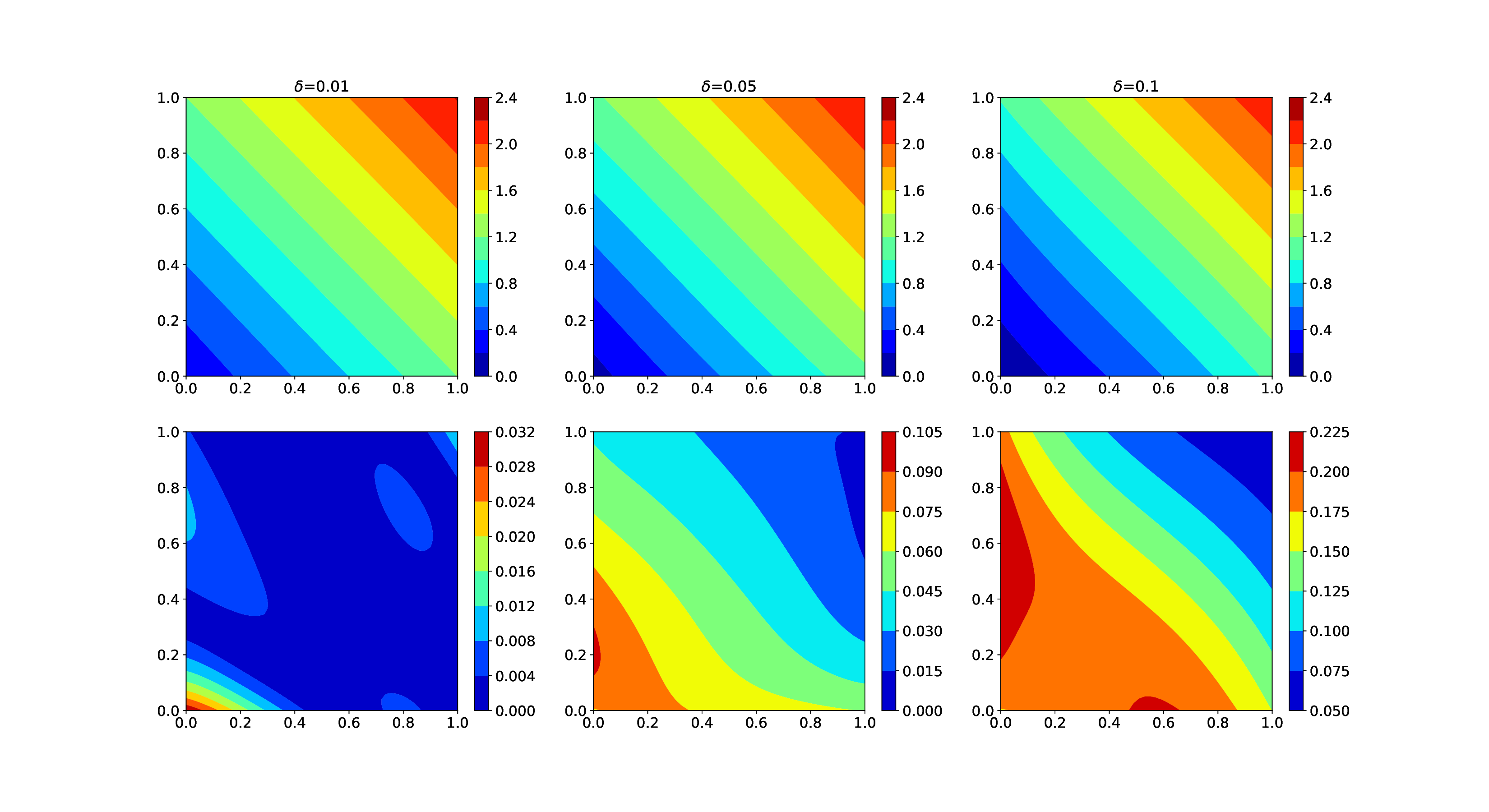}
\caption{The reconstruction of potential function $q$ (first line), absolute pointwise error $|q-q^*|$ (second line) for various noise levels.}\label{PINN-q-3D}
\end{figure}

Finally, we conduct the experiments to evaluate the robustness of the proposed scheme in terms of network structure (number of hidden layers and per-layer nurons).
More specifically, we run Algorithm \ref{alg1} with per-layer neuron numbers $\text{NN}=5,10,20,25$ for fixed hidden layers $\text{NL}=6$ and with hidden layers $\text{NL}=4,6,8,10$ for fixed per-layer neuron numbers $\text{NN}=25$, respectively. The reconstruction errors are presented in Figure \ref{PINN-Robust}. It also seems that larger NL and/or larger NN neural networks yield lower reconstruction error. In this example, for fixed hidden layer $\text{NL}=6$, we test the per-layer neuron numbers $\text{NN}\geq 30$ and find that there will be bad reconstruction result, that is, the relative error for $q$ with $\text{NL}=6, \text{NN}=30$ is larger than the case $\text{NN}=10,20,25$ as the iterations increasing. Therefore, more layers and/or neurons with much more parameters to train, yield longer training time, and may result in overfitting of the reconstruction.

\begin{figure}[h!]
\centering
\includegraphics[width=1.0\textwidth,height=0.25\textheight,center]{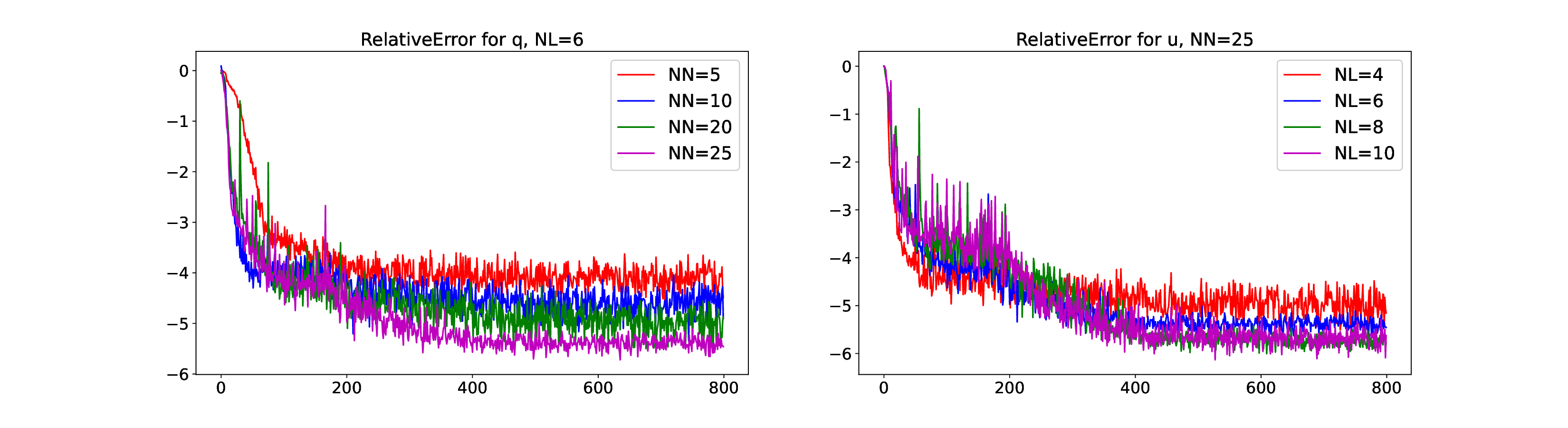}
\caption{The reconstruction error of $q$ under various per-layer neuron numbers with fixed layer numbers NL=6 (left), the reconstruction error of $q$ under various layer numbers with fixed per-layer neuron number NN=25 (right).}\label{PINN-Robust}
\end{figure}

\section{Concluding remarks.}\label{Sec4}
In this work, a deep neural network-based reconstruction scheme has been proposed to solve an inverse potential problem in the parabolic equation. The proposed method has shown superior performance in high-dimensional space.
We prove the uniqueness of the inverse potential problem. A new loss function has been introduced, which includes regularization terms that depend on the derivatives of the residuals for both the partial differential equation and the measurement data. These regularization terms aim to address the ill-posedness of the inverse problem and enhance the regularity of the solution. Additionally, the mollification method has been employed to improve the regularity of the noisy data, where it can reduce the perturbation errors caused by numerical differentiation on the noisy data. Generalization estimates based on the conditional stability of linear inverse source problems and the mollification error estimate on noisy data have been established, which provide a measure of the stability and accuracy of the proposed method in solving the inverse potential problem. Numerical experiments have been conducted to evaluate the performance of proposed method, which indicate the efficiency the approach in this work.

\section*{Acknowledgments}

Mengmeng Zhang is supported by Foundation of Hebei University of Technology (Grant No.282022550) and Foundation of Tianjin Education Commission Research Program(Grant No.2022KJ102). Zhidong Zhang is supported by National Natural Science Foundation of China (Grant No. 12101627).

\bibliographystyle{plainurl} 
\bibliography{main}

\begin{thebibliography}{10}

\bibitem{Bao2}
Gang Bao, Xiaojing Ye, Yaohua Zang, and Haomin Zhou.
\newblock Numerical solution of inverse problems by weak adversarial networks.
\newblock {\em Inverse Problems}, 36(11):115003, 31, 2020.
\newblock URL: \url{https://doi.org/10.1088/1361-6420/abb447}, \href
  {http://dx.doi.org/10.1088/1361-6420/abb447}
  {\path{doi:10.1088/1361-6420/abb447}}.

\bibitem{Lesnic}
K.~Cao and D.~Lesnic.
\newblock Reconstruction of the space-dependent perfusion coefficient from
  final time or time-average temperature measurements.
\newblock {\em J. Comput. Appl. Math.}, 337:150--165, 2018.
\newblock URL: \url{https://doi.org/10.1016/j.cam.2018.01.010}, \href
  {http://dx.doi.org/10.1016/j.cam.2018.01.010}
  {\path{doi:10.1016/j.cam.2018.01.010}}.

\bibitem{ChenDH}
De-Han Chen, Daijun Jiang, and Jun Zou.
\newblock Convergence rates of {T}ikhonov regularizations for elliptic and
  parabolic inverse radiativity problems.
\newblock {\em Inverse Problems}, 36(7):075001, 21, 2020.
\newblock URL: \url{https://doi.org/10.1088/1361-6420/ab8449}, \href
  {http://dx.doi.org/10.1088/1361-6420/ab8449}
  {\path{doi:10.1088/1361-6420/ab8449}}.

\bibitem{Chen}
Qun Chen and Jijun Liu.
\newblock Solving an inverse parabolic problem by optimization from final
  measurement data.
\newblock {\em J. Comput. Appl. Math.}, 193(1):183--203, 2006.
\newblock URL: \url{https://doi.org/10.1016/j.cam.2005.06.003}, \href
  {http://dx.doi.org/10.1016/j.cam.2005.06.003}
  {\path{doi:10.1016/j.cam.2005.06.003}}.

\bibitem{Chen-Lu}
Yuyao Chen, Lu~Lu, George~Em Karniadakis, and Luca~Dal Negro.
\newblock Physics-informed neural networks for inverse problems in nano-optics
  and metamaterials.
\newblock {\em Opt. Express}, 28(8):11618--11633, 2020.
\newblock URL: \url{https://opg.optica.org/oe/abstract.cfm?URI=oe-28-8-11618},
  \href {http://dx.doi.org/10.1364/OE.384875} {\path{doi:10.1364/OE.384875}}.

\bibitem{Choulli}
Mourad Choulli and Masahiro Yamamoto.
\newblock Generic well-posedness of an inverse parabolic problem---the
  {H}\"{o}lder-space approach.
\newblock {\em Inverse Problems}, 12(3):195--205, 1996.
\newblock URL: \url{https://doi.org/10.1088/0266-5611/12/3/002}, \href
  {http://dx.doi.org/10.1088/0266-5611/12/3/002}
  {\path{doi:10.1088/0266-5611/12/3/002}}.

\bibitem{Choulli1}
Mourad Choulli and Masahiro Yamamoto.
\newblock An inverse parabolic problem with non-zero initial condition.
\newblock {\em Inverse Problems}, 13(1):19--27, 1997.
\newblock URL: \url{https://doi.org/10.1088/0266-5611/13/1/003}, \href
  {http://dx.doi.org/10.1088/0266-5611/13/1/003}
  {\path{doi:10.1088/0266-5611/13/1/003}}.

\bibitem{Deng1}
Zui-Cha Deng, Liu Yang, and Jian-Ning Yu.
\newblock Identifying the radiative coefficient of heat conduction equations
  from discrete measurement data.
\newblock {\em Appl. Math. Lett.}, 22(4):495--500, 2009.
\newblock URL: \url{https://doi.org/10.1016/j.aml.2008.06.023}, \href
  {http://dx.doi.org/10.1016/j.aml.2008.06.023}
  {\path{doi:10.1016/j.aml.2008.06.023}}.

\bibitem{Duan_2022}
Chenguang Duan, Yuling Jiao, Yanming Lai, Dingwei Li, Xiliang Lu, and
  Jerry~Zhijian Yang.
\newblock Convergence rate analysis for deep {R}itz method.
\newblock {\em Commun. Comput. Phys.}, 31(4):1020--1048, 2022.
\newblock URL: \url{https://doi.org/10.4208/cicp.oa-2021-0195}, \href
  {http://dx.doi.org/10.4208/cicp.oa-2021-0195}
  {\path{doi:10.4208/cicp.oa-2021-0195}}.

\bibitem{EWN-DRM}
Weinan E and Bing Yu.
\newblock The deep {R}itz method: a deep learning-based numerical algorithm for
  solving variational problems.
\newblock {\em Commun. Math. Stat.}, 6(1):1--12, 2018.
\newblock URL: \url{https://doi.org/10.1007/s40304-018-0127-z}, \href
  {http://dx.doi.org/10.1007/s40304-018-0127-z}
  {\path{doi:10.1007/s40304-018-0127-z}}.

\bibitem{Evans}
Lawrence~C. Evans.
\newblock {\em Partial differential equations}, volume~19 of {\em Graduate
  Studies in Mathematics}.
\newblock American Mathematical Society, Providence, RI, second edition, 2010.
\newblock URL: \url{https://doi.org/10.1090/gsm/019}, \href
  {http://dx.doi.org/10.1090/gsm/019} {\path{doi:10.1090/gsm/019}}.

\bibitem{Goodfellow-et-al-2016}
Ian Goodfellow, Yoshua Bengio, and Aaron Courville.
\newblock {\em Deep Learning}.
\newblock MIT Press, 2016.
\newblock \url{http://www.deeplearningbook.org}.

\bibitem{Hong_2021}
Qingguo Hong, Jonathan W.~Siegel, and Jinchao Xu.
\newblock Rademacher complexity and numerical quadrature analysis of stable
  neural networks with applications to numerical pdes.
\newblock {\em arXiv preprint}.
\newblock \href {http://dx.doi.org/arXiv:2104.02903}
  {\path{doi:arXiv:2104.02903}}.

\bibitem{Isakov2}
Victor Isakov.
\newblock {\em Inverse problems for partial differential equations}, volume 127
  of {\em Applied Mathematical Sciences}.
\newblock Springer, Cham, third edition, 2017.
\newblock URL: \url{https://doi.org/10.1007/978-3-319-51658-5}, \href
  {http://dx.doi.org/10.1007/978-3-319-51658-5}
  {\path{doi:10.1007/978-3-319-51658-5}}.

\bibitem{Jiao}
Yuling Jiao, Yanming Lai, Dingwei Li, Xiliang Lu, Fengru Wang, Yang Wang, and
  Jerry~Zhijian Yang.
\newblock A rate of convergence of physics informed neural networks for the
  linear second order elliptic {PDE}s.
\newblock {\em Commun. Comput. Phys.}, 31(4):1272--1295, 2022.
\newblock URL: \url{https://doi.org/10.4208/cicp.oa-2021-0186}, \href
  {http://dx.doi.org/10.4208/cicp.oa-2021-0186}
  {\path{doi:10.4208/cicp.oa-2021-0186}}.

\bibitem{Jin1}
Bangti Jin, Xiliang Lu, Qimeng Quan, and Zhi Zhou.
\newblock Convergence rate analysis of {G}alerkin approximation of inverse
  potential problem.
\newblock {\em Inverse Problems}, 39(1):Paper No. 015008, 26, 2023.
\newblock URL: \url{https://doi.org/10.1088/1361-6420/aca70e}, \href
  {http://dx.doi.org/10.1088/1361-6420/aca70e}
  {\path{doi:10.1088/1361-6420/aca70e}}.

\bibitem{Kaltenbacher2}
Barbara Kaltenbacher and William Rundell.
\newblock Recovery of multiple coefficients in a reaction-diffusion equation.
\newblock {\em Journal of Mathematical Analysis and Applications},
  481(1):123475, 2020.
\newblock URL:
  \url{https://www.sciencedirect.com/science/article/pii/S0022247X19307437},
  \href {http://dx.doi.org/https://doi.org/10.1016/j.jmaa.2019.123475}
  {\path{doi:https://doi.org/10.1016/j.jmaa.2019.123475}}.

\bibitem{Kamynin}
V.~L. Kamynin and A.~B. Kostin.
\newblock Two inverse problems of the determination of a coefficient in a
  parabolic equation.
\newblock {\em Differ. Uravn.}, 46(3):372--383, 2010.
\newblock URL: \url{https://doi.org/10.1134/S0012266110030079}, \href
  {http://dx.doi.org/10.1134/S0012266110030079}
  {\path{doi:10.1134/S0012266110030079}}.

\bibitem{Lu-DeepONet-2}
Lu~Lu, Pengzhan Jin, Guofei Pang, Zhongqiang Zhang, and George~Em Karniadakis.
\newblock Learning nonlinear operators via deeponet based on the universal
  approximation theorem of operators.
\newblock {\em Nature machine intelligence}, 3(3):218--229, 2021.

\bibitem{Lu-DeepXDE}
Lu~Lu, Xuhui Meng, Zhiping Mao, and George~Em Karniadakis.
\newblock Deep{XDE}: a deep learning library for solving differential
  equations.
\newblock {\em SIAM Rev.}, 63(1):208--228, 2021.
\newblock URL: \url{https://doi.org/10.1137/19M1274067}, \href
  {http://dx.doi.org/10.1137/19M1274067} {\path{doi:10.1137/19M1274067}}.

\bibitem{Lu_2021}
Yulong Lu, Jianfeng Lu, and Min Wang.
\newblock A priori generalization analysis of the deep ritz method for solving
  high dimensional elliptic partial differential equations.
\newblock volume 134 of {\em Proceedings of Machine Learning Research}, pages
  3196--3241. PMLR, 2021.
\newblock URL: \url{https://proceedings.mlr.press/v134/lu21a.html}.

\bibitem{Luo_2020}
Tao Luo and Haizhao Yang.
\newblock Two-layer neural networks for partial differential equations:
  Optimization and generalization theory, 2020.
\newblock \href {http://arxiv.org/abs/2006.15733} {\path{arXiv:2006.15733}}.

\bibitem{Mishra2}
Siddhartha Mishra and Roberto Molinaro.
\newblock Estimates on the generalization error of physics-informed neural
  networks for approximating a class of inverse problems for {PDE}s.
\newblock {\em IMA J. Numer. Anal.}, 42(2):981--1022, 2022.
\newblock URL: \url{https://doi.org/10.1093/imanum/drab032}, \href
  {http://dx.doi.org/10.1093/imanum/drab032}
  {\path{doi:10.1093/imanum/drab032}}.

\bibitem{Prilepko}
Aleksei~Ivanovich Prilepko and Andrew~Borisovich Kostin.
\newblock On certain inverse problems for parabolic equations with final and
  integral observation.
\newblock {\em Matematicheskii Sbornik}, 183(4):49--68, 1992.
\newblock URL: \url{https://dx.doi.org/10.1070/SM1993v075n02ABEH003394}, \href
  {http://dx.doi.org/10.1070/SM1993v075n02ABEH003394}
  {\path{doi:10.1070/SM1993v075n02ABEH003394}}.

\bibitem{Prilepko1}
Aleksey~I. Prilepko, Dmitry~G. Orlovsky, and Igor~A. Vasin.
\newblock {\em Methods for solving inverse problems in mathematical physics},
  volume 231 of {\em Monographs and Textbooks in Pure and Applied Mathematics}.
\newblock Marcel Dekker, Inc., New York, 2000.

\bibitem{Raissi}
M.~Raissi, P.~Perdikaris, and G.~E. Karniadakis.
\newblock Physics-informed neural networks: a deep learning framework for
  solving forward and inverse problems involving nonlinear partial differential
  equations.
\newblock {\em J. Comput. Phys.}, 378:686--707, 2019.
\newblock URL: \url{https://doi.org/10.1016/j.jcp.2018.10.045}, \href
  {http://dx.doi.org/10.1016/j.jcp.2018.10.045}
  {\path{doi:10.1016/j.jcp.2018.10.045}}.

\bibitem{Raissi-Yazdani}
Maziar Raissi, Alireza Yazdani, and George~Em Karniadakis.
\newblock Hidden fluid mechanics: A navier-stokes informed deep learning
  framework for assimilating flow visualization data, 2018.
\newblock \href {http://arxiv.org/abs/1808.04327} {\path{arXiv:1808.04327}}.

\bibitem{Rundell}
William Rundell.
\newblock The determination of a parabolic equation from initial and final
  data.
\newblock {\em Proc. Amer. Math. Soc.}, 99(4):637--642, 1987.
\newblock URL: \url{https://doi.org/10.2307/2046467}, \href
  {http://dx.doi.org/10.2307/2046467} {\path{doi:10.2307/2046467}}.

\bibitem{de2023error}
Tim~De Ryck, Ameya~D. Jagtap, and Siddhartha Mishra.
\newblock Error estimates for physics informed neural networks approximating
  the navier-stokes equations, 2023.
\newblock \href {http://arxiv.org/abs/2203.09346} {\path{arXiv:2203.09346}}.

\bibitem{Shin1}
Yeonjong Shin, J\'{e}r\^{o}me Darbon, and George~Em Karniadakis.
\newblock On the convergence of physics informed neural networks for linear
  second-order elliptic and parabolic type {PDE}s.
\newblock {\em Commun. Comput. Phys.}, 28(5):2042--2074, 2020.
\newblock URL: \url{https://doi.org/10.4208/cicp.oa-2020-0193}, \href
  {http://dx.doi.org/10.4208/cicp.oa-2020-0193}
  {\path{doi:10.4208/cicp.oa-2020-0193}}.

\bibitem{Shin2}
Yeonjong Shin, Zhongqiang Zhang, and George~Em Karniadakis.
\newblock Error estimates of residual minimization using neural networks for
  linear pdes, 2020.
\newblock \href {http://arxiv.org/abs/2010.08019} {\path{arXiv:2010.08019}}.

\bibitem{Shukla}
Khemraj Shukla, Patricio~Clark Di~Leoni, James Blackshire, Daniel Sparkman, and
  George~Em Karniadakis.
\newblock Physics-informed neural network for ultrasound nondestructive
  quantification of surface breaking cracks.
\newblock {\em Journal of Nondestructive Evaluation}, 39:1--20, 2020.

\bibitem{Sirignano-DGM}
Justin Sirignano and Konstantinos Spiliopoulos.
\newblock D{GM}: a deep learning algorithm for solving partial differential
  equations.
\newblock {\em J. Comput. Phys.}, 375:1339--1364, 2018.
\newblock URL: \url{https://doi.org/10.1016/j.jcp.2018.08.029}, \href
  {http://dx.doi.org/10.1016/j.jcp.2018.08.029}
  {\path{doi:10.1016/j.jcp.2018.08.029}}.

\bibitem{Lesnic1}
Dumitru Trucu, Derek~B. Ingham, and Daniel Lesnic.
\newblock Space-dependent perfusion coefficient identification in the transient
  bio-heat equation.
\newblock {\em J. Engrg. Math.}, 67(4):307--315, 2010.
\newblock URL: \url{https://doi.org/10.1007/s10665-009-9319-6}, \href
  {http://dx.doi.org/10.1007/s10665-009-9319-6}
  {\path{doi:10.1007/s10665-009-9319-6}}.

\bibitem{zou1}
Masahiro Yamamoto and Jun Zou.
\newblock Simultaneous reconstruction of the initial temperature and heat
  radiative coefficient.
\newblock volume~17, pages 1181--1202. 2001.
\newblock Special issue to celebrate Pierre Sabatier's 65th birthday
  (Montpellier, 2000).
\newblock URL: \url{https://doi.org/10.1088/0266-5611/17/4/340}, \href
  {http://dx.doi.org/10.1088/0266-5611/17/4/340}
  {\path{doi:10.1088/0266-5611/17/4/340}}.

\bibitem{Bao1}
Yaohua Zang, Gang Bao, Xiaojing Ye, and Haomin Zhou.
\newblock Weak adversarial networks for high-dimensional partial differential
  equations.
\newblock {\em J. Comput. Phys.}, 411:109409, 14, 2020.
\newblock URL: \url{https://doi.org/10.1016/j.jcp.2020.109409}, \href
  {http://dx.doi.org/10.1016/j.jcp.2020.109409}
  {\path{doi:10.1016/j.jcp.2020.109409}}.

\bibitem{Zhang-Li-Liu}
Mengmeng Zhang, Qianxiao Li, and Jijun Liu.
\newblock On stability and regularization for data-driven solution of parabolic
  inverse source problems.
\newblock {\em J. Comput. Phys.}, 474:Paper No. 111769, 20, 2023.
\newblock URL: \url{https://doi.org/10.1016/j.jcp.2022.111769}, \href
  {http://dx.doi.org/10.1016/j.jcp.2022.111769}
  {\path{doi:10.1016/j.jcp.2022.111769}}.

\bibitem{Zhangliu}
Mengmeng Zhang and Jijun Liu.
\newblock Recovery of non-smooth radiative coefficient from nonlocal
  observation by diffusion system.
\newblock {\em J. Inverse Ill-Posed Probl.}, 28(3):389--410, 2020.
\newblock URL: \url{https://doi.org/10.1515/jiip-2019-0029}, \href
  {http://dx.doi.org/10.1515/jiip-2019-0029}
  {\path{doi:10.1515/jiip-2019-0029}}.

\bibitem{Zhang-Zhou}
Zhengqi Zhang, Zhidong Zhang, and Zhi Zhou.
\newblock Identification of potential in diffusion equations from terminal
  observation: analysis and discrete approximation.
\newblock {\em SIAM J. Numer. Anal.}, 60(5):2834--2865, 2022.
\newblock URL: \url{https://doi.org/10.1137/21M1446708}, \href
  {http://dx.doi.org/10.1137/21M1446708} {\path{doi:10.1137/21M1446708}}.

\end{thebibliography}

\end{document}